\documentclass[A4, 12pt]{amsart}

\usepackage{amsmath}    
\usepackage{hyperref}  
\usepackage{fullpage}
\usepackage{bbm}
\usepackage{dsfont}
\usepackage{amssymb}
\usepackage{mathrsfs}
\usepackage{enumerate}

\newtheorem{theorem}{Theorem}
\newtheorem{proposition}{Proposition}
\newtheorem{lemma}{Lemma}
\newtheorem{corollary}{Corollary}

\newcommand{\half}{\ensuremath{\frac{1}{2}}}
\newcommand{\Q}{\mathbb{Q}}
\newcommand{\Z}{\mathbb{Z}}
\newcommand{\R}{\mathbb{R}}
\newcommand{\C}{\mathbb{C}}
\newcommand{\N}{\mathbb{N}}

\newcommand{\real}{\mathop{\rm Re}}
\newcommand{\imag}{\mathop{\rm Im}}
\newcommand{\sgn}{\mathop{\rm sgn}}

\newcommand{\eqdef}{\mathop{=}^{\rm def}}

\newcommand{\eps}{\varepsilon}
\newcommand{\es}[1]{\begin{equation}\begin{split}#1\end{split}\end{equation}}
\newcommand{\est}[1]{\begin{equation*}\begin{split}#1\end{split}\end{equation*}}

\newcommand{\psum}{\mathop{\sum\nolimits^+}}
\newcommand{\msum}{\mathop{\sum\nolimits^-}}
\newcommand{\pmsum}{\mathop{\sum\nolimits^\pm}}

\title{A Twisted Motohashi Formula and Weyl-Subconvexity for $L$-functions of Weight Two Cusp Forms}
\author{Ian Petrow}
\thanks{The author is partially supported by Swiss National Science Foundation grant 200021\_137488}
\email{ian.petrow@epfl.ch}
\address{\'Ecole Polytechnique F\'ed\'erale de Lausanne \\ Section de Math\'ematiques \\ 1015 Lausanne, Switzerland}

\begin{document}
\maketitle

\vspace{-24pt}

\begin{abstract}
We derive a Motohashi-type formula for the cubic moment of central values of $L$-functions of level $q$ cusp forms twisted by quadratic characters of conductor $q$, previously studied by Conrey and Iwaniec and Young.  Corollaries of this formula include Weyl-subconvex bounds for $L$-functions of weight two cusp forms twisted by quadratic characters, and estimates towards the Ramanujan-Petersson conjecture for Fourier coefficients of weight 3/2 cusp forms. 
\end{abstract}

\section{Introduction}\label{intro}

Let $f$ be a classical holomorphic cusp form of even positive weight $\kappa$, odd square-free level $q$ and trivial central character.  Let $\chi$ the unique primitive Dirichlet character of conductor $q$ corresponding to a quadratic field extension $K/\Q$.  Consider the value of the $L$-function of $f\otimes \chi$ at its center point of symmetry: $L(1/2,f\otimes \chi)$.  The Waldspurger formula \cite{Waldspurger} gives an arithmetic interpretation to this positive real number.  

In this paper we study the below cubic moment of central values of $L$-functions, first considered by Conrey and Iwaniec in \cite{ConreyIwaniec}.  For $\kappa \geq 12$ and any $\eps>0$ they prove that \es{\label{CIresult}\sum_{f \in \mathscr{F}_\kappa(q)}L(1/2,f\otimes \chi)^3 \ll_{\kappa,\eps} q^{1+\eps}.} The Lindel\"of-on-average estimate \eqref{CIresult} stands out among the vast literature on moments of $L$-functions because it goes far beyond what one expects to be provable using the current technology.    
Conrey and Iwaniec obtain as corollaries of their estimate Weyl-subconvex bounds for several important families of $L$-functions, and the best currently-known bound on Fourier coefficients of $1/2$-integral weight modular forms.  

Here we revisit the cubic moment of Conrey and Iwaniec and derive a corresponding dual moment in Theorem \ref{Thm2} which is reminiscent of some formulas first derived by Motohashi.  These Motohashi formulas relate the spectral cubic moment of (un-twisted) ${\rm GL}_2$ $L$-functions to an average of four Riemann zeta functions in $t$-aspect.  See Motohashi chapter four \cite{Motohashi}, or also Michel and Venkatesh sections 4.5.4 and 4.5.5 of \cite{MichelVenkateshGL2}.  In Theorem \ref{Thm2} we give a twisted Motohashi-type formula for the dual sums of \eqref{CIresult}.  This formula does not seem to follow in a straightforward way from the general arguments of Michel and Venkatesh.  

The Motohashi-type formula of Theorem \ref{Thm2} is crucial in extending \eqref{CIresult} to small weights $\kappa\geq 2$.  We give our improved estimate for the cubic moment in Theorem \ref{MP} and Weyl-subconvex bounds for the central values of $L$-functions in Corollary \ref{cor1}.  Previously, the best available estimates for small weights $\kappa\geq 2$ were apparently special cases of the results of Blomer and Harcos \cite{BlomerHarcosBurgessSC,BlomerHarcosBurgessSCad}, which are quite general and of Burgess quality in $q$.  The present paper is in some sense a counterpart to the work of Young  \cite{YoungCubic} who gives estimates for the same cubic moment which are uniform as $\kappa\rightarrow \infty$.  

Additionally, the Motohashi-type formula allows us to replace the epsilons appearing in the previous results \cite{ConreyIwaniec,YoungCubic} with explicit powers of $\log q$ and the divisor function of $q$.  This gives the best currently-know estimates.  As in the original work of Conrey and Iwaniec we focus on the case of holomorphic forms, but our results carry over to the case of non-holomorphic Maass waveforms and Eisenstein series as well.  

Now we describe our results more precisely.  Let $\lambda_f(n)$ denote the Hecke eigenvalues of $f$ normalized so that $|\lambda_f(n)| \leq d(n)$ and let $\mathscr{F}_\kappa(q)$ be an orthonormal basis (with respect to the Petersson inner product) of Hecke eigenforms .  Define for $\real(s)>1$ the $L$-series \est{ L(s,f \otimes \chi) \eqdef \sum_{n=1}^\infty \frac{\lambda_f(n) \chi(n)}{n^s},} the local $L$-function at the infinite place, and the completed $L$-function \est{L_\infty(s,f \otimes \chi) \eqdef  \left(\frac{q}{2\pi}\right)^{s-1/2}\Gamma\left(s + \frac{\kappa-1}{2}\right) \\ \Lambda(s,f\otimes \chi)\eqdef L_\infty(s,f\otimes \chi)L(s,f\otimes \chi).} 
We assume that $i^\kappa = \chi(-1)$ so that the sign of the functional equation is fixed to $+1$: \est{\Lambda(s,f\otimes \chi) = \Lambda(1-s,f \otimes \chi).} 
Let the Fourier coefficients of $f$ be $a_f(n)=a_f(1)\lambda_f(n)$, and let \est{\omega_f = (4\pi)^{1-\kappa} \Gamma(\kappa-1)|a_f(1)|^2} be the standard harmonic weights necessary for the clean application of the Petersson trace formula.  These weights do not vary much, in fact in the case of holomorphic forms we have \es{\label{weight size}\frac{1}{\kappa (q+1) (\log \kappa q)^3} \ll \omega_f \ll \frac{\log \kappa q +1}{\kappa (q+1)},} by \cite{DeWC, CarlettiMontiPerelli, HLAppendix}.  The weight $\kappa$ is always considered fixed, and all implicit constants may depend on $\kappa$. In this paper we prove the following refinement of \eqref{CIresult}.

\begin{theorem}\label{MP}
Suppose $\kappa \geq 2$ and $q$ odd square-free with all $\real(\alpha_i) \ll 1/\log q$.  Let $\nu(q)$ denote the number of prime factors of $q$.  There exists an absolute constant $C>0$ for which \est{\sum_{f\in \mathscr{F}_\kappa(q) }\omega_f \Lambda(\half+\alpha_1,f\otimes \chi) \Lambda(1/2+\alpha_2,f\otimes \chi)\Lambda(1/2+\alpha_3,f\otimes \chi) \\ \ll C^{\nu(q)} (\log (\nu(q)+1))^{2+4/C}  \begin{cases} (\log q)^{4}  & \text{Case A} \\ (\log q)^{3} |\zeta(1+2i|\alpha_i|)| & \text{Case B} \\ (\log q)^{2}|\zeta(1+i|\alpha_i|+i|\alpha_j|)||\zeta(1+i|\alpha_i|-i|\alpha_j|)| & \text{Case C}. \end{cases} } where \est{\begin{cases}\text{Case A:} & \text{all } |\alpha_i| \leq 1/\log q \\ \text{Case B:} & \text{for all } i,j,\,\, ||\alpha_i|-|\alpha_j|| \leq 1/\log q \text{ and } |\alpha_i|>1/\log q \\ \text{Case C:} & \text{there exists } i \neq j \text{ such that } ||\alpha_i|-|\alpha_j||>1/\log q. \end{cases}} If $\kappa=2$ then the powers of $\log q$ above are increased to $5,4,3$ in cases A,B,C, respectively. \end{theorem} 

The most interesting case of Theorem \ref{MP} is of course $\alpha_1=\alpha_2=\alpha_3=0$.  The local $L$-function at the archimedian place $L_\infty$ is constant across the family, and so Theorem \ref{MP} and \eqref{weight size} immediately give a new estimate for \eqref{CIresult} without epsilons and which is valid for all $\kappa \geq 2$. 

The main new idea which leads us to Theorem \ref{MP} is a ``Motohashi-type formula'' for the dual sums produced by applying trace formulas.  The moment under study in the above proposition breaks up into the expected main term, the dual moment described in the next proposition, and a small error term.  In Theorem \ref{Thm2} we only write the dual moment for prime levels $q$ for ease of exposition.  

\begin{theorem}[Motohashi-type Formula]\label{Thm2}  Let $\kappa \geq 2$, $q$ prime, and $\alpha_1,\alpha_2,\alpha_3$ be three complex numbers with $\real(\alpha_i)< 1/\log q$.  The group $(\Z/2\Z)^3$ acts on the set of triples $\alpha=(\alpha_1,\alpha_2,\alpha_3)$ by multiplication by $\pm1$ on each entry.  Let $\psum$ and $\msum$ denote sums over the primitive even, resp. odd, Dirichlet characters of conductor $q$, $\eps(\psi)$ be the sign of the functional equation of $L(s,\psi)$, and $\varphi(q)$ be the number of primitive Dirichlet characters of conductor $q$.  

For $\chi$ the quadratic character of prime conductor $q$ let $g(\chi,\psi)$ be the character sum \est{g(\chi,\psi) = \sum_{u,v \pmod q} \chi(uv(u+1)(v+1))\psi(uv-1)} which satisfies $|g(\chi,\psi)/q|\leq C$ for an absolute constant $C$.  We have that \est{\sum_{f\in \mathscr{F}_\kappa(q) }\omega_f \Lambda(\half+\alpha_1,f\otimes \chi) \Lambda(1/2+\alpha_2,f\otimes \chi)\Lambda(1/2+\alpha_3,f\otimes \chi) \\  = \sum_{\sigma \in (\Z/2\Z)^3} MT(\mathscr{F},\sigma\alpha) + DM(\mathscr{F},\sigma\alpha) + O(q^{-1/3+\eps}),} where \est{MT(\mathscr{F},\alpha) \eqdef & \zeta_q(1+\alpha_1+\alpha_2)\zeta_q(1+\alpha_2+\alpha_3)\zeta_q(1+\alpha_3+\alpha_1)\\ & \times L_\infty(1/2+\alpha_1,f \otimes \chi)L_\infty(1/2+\alpha_2,f \otimes \chi)L_\infty(1/2+\alpha_3,f \otimes \chi)} and \est{DM(\mathscr{F},\alpha) \eqdef \sum_{\pm}\frac{1}{(2\pi i )^4} \int_{(1/2)} \iiint_{(\eps)} \frac{U^\pm(s,u_1,u_2,u_3)q^{2s+u_1+u_2+u_3-1}}{(u_1-\alpha_1)(u_2-\alpha_2)(u_3-\alpha_3)} \\ \times \frac{1}{\varphi(q)}\pmsum_{\psi \pmod q} \frac{g(\chi,\psi)}{q} \overline{\eps(\psi)}^2 L(s,\psi)L(s+u_1+u_2,\psi)L(s+u_2+u_3,\psi)L(s+u_3+u_1,\psi)\\ \,du_1\,du_2\,du_3\,ds.}  The two functions $U^\pm(s,u_1,u_2,u_3)$ are each holomorphic in the region \est{\max(\real(u_i))-1/2 < \real(s+u_1+u_2+u_3) < \kappa/2 \\ -\kappa/2 < \real(u_i),} symmetric in the $u_i$ variables, and satisfy the bounds \est{U^\pm \ll_\eps (1+|s|)^{\ell} \exp(-(\pi/2-\eps)|\imag(u_1+u_2+u_3)|)} for any $\eps>0$, where \est{\ell = \max\left(\real(s+u_1+u_2+u_3) -(\kappa+1)/2,-3/2\right).} 
\end{theorem}

Moving the $u_i$ contours to the lines $\real(u_i) = 1/\log q$ and applying the multiplicative large sieve (see section \ref{Lsieve}) one derives Theorem \ref{MP} from Theorem \ref{Thm2}. 

As predicted by the generalized Riemann hypothesis, the values of these $L$-functions at $s=1/2$ are in fact known to satisfy \es{\label{walds}L(1/2,f\otimes \chi) \geq 0} due to the well-known result of Waldspurger \cite{Waldspurger}.  See also the classical work-out as an explicit formula for full level due to Kohnen and Zagier \cite{KZ}, and for general level due to Kohnen \cite{KohnenGeneralLevel}.  As a consequence of positivity \eqref{walds}, the bounds on harmonic weights \eqref{weight size}, and our main result Theorem \ref{MP} we derive the following strengthened form of the subconvex bound found in Conrey and Iwaniec as their Corollary 1.2.

\begin{corollary}\label{cor1}
Let $f$ be a primitive holomorphic cusp form of weight $\kappa \geq 2$ with level dividing $q$, and let $\chi \pmod q$ be the quadratic character of conductor $q$. Then \est{L(1/2,f \otimes \chi) \ll_\kappa q^{1/3} (\log q)^{7/3} C^{\nu(q)/3}(\log (\nu(q)+1))^{(2+4/C)/3}.}  If $\kappa = 2$ the  power of $\log q$ above is increased to $8/3$.
\end{corollary}

Consider the case that $f$ is of level $q$ and weight $\kappa\geq 2$.  Then $f$ corresponds under the Shimura lift to a half-integral cusp form $F$ of weight $(\kappa+1)/2$ and level $4q$.  We write the Fourier expansion as \est{F(z) = \sum_{n=1}^\infty c_F(n) n^{(k-1)/4}e(nz),} so that the Ramanujan-Petersson conjecture gives that the Fourier coefficients are $c_F(n) \ll_\eps n^\eps$, for $n$ odd square-free and $(q,n)=1$.  Via e.g. Corollary 1 of Kohnen \cite{KohnenGeneralLevel} we derive the following estimate.  
\begin{corollary}\label{FCcor}
Let $F$ be a level $4q$ half-integral weight $(\kappa+1)/2$ cusp form with $\kappa\geq 2$.  If $n$ odd square-free and relatively prime to $q$ with $\chi_n(-1) = i^\kappa$ then \est{c_F(n) \ll_F n^{1/6}( \log n )^{7/6} C^{\nu(n)/6}(\log (\nu(n)+1))^{(2+4/C)/3}.}  If $\kappa = 2$ the power of $\log n$ above is increased to $4/3$.
\end{corollary}

Corollary 2 has applications to the rate of equidistribution of integral points on ellipsoids, including the most interesting case of those lying in $\R^3$.  See Iwaniec \cite{ClassIw} chapter 11.

As in Conrey and Iwaniec's original paper \cite{ConreyIwaniec} we have given complete proofs only in the case of holomorphic forms, as the most interesting application of the Motohashi-like formula in Theorem \ref{Thm2} is the case of small weight $\kappa=2$.  Nonetheless, the proofs should carry over to the case of Maass forms of weight 0 and Eisenstein series by replacing the Petersson formula with the Kuznetsov formula.  Making this substitution changes the $J$-Bessel function to a more general integral transform of the chosen weight function on the spectral side.  The formula \eqref{Jintegralformula} for the $J$-Bessel function that we use in section \ref{sec:statphase} has an analogue which is needed in the Kuznetsov case, in which the interior $\sin$ in \eqref{Jintegralformula} is replaced by any of $\pm\{ \cos, \sin, \cosh,\sinh\}$.  See the work of Young \cite{YoungCubic} where a similar stationary phase argument is carried out in the generality needed for the Kuznetsov formula.  

\section*{Acknowledgement}
I would like to express my deep appreciation to Philippe Michel, Paul Nelson, Kannan Soundararajan, Akshay Venkatesh and Matthew Young for many helpful discussions, and to thank the \'Ecole Polytechnique F\'ed\'erale de Lausanne and the Swiss National Fund for their generous financial support.

\section{Standard Initial Steps}\label{Initial}

Let \es{\label{Vdef}V_{1/2+\alpha_i}(x) \eqdef \frac{1}{2\pi i} \int_{(2)} \frac{\Gamma(s+\kappa/2) }{(s-\alpha_i)} (2\pi)^{-s} x^{-s} \,ds} and \est{\Lambda_0(\alpha_i,\chi) \eqdef \sum_{n=1}^\infty \frac{\lambda_f(n)\chi(n)}{n^{1/2}} V_{1/2+\alpha_i}(n/q).}  We then have a standard approximate functional equation.  \begin{proposition}[Approximate Functional Equation]\label{AFE} We have \est{\Lambda(1/2+\alpha_i,f\otimes \chi) = \Lambda_0(\alpha_i,\chi) +\Lambda_0(-\alpha_i,\chi).}\end{proposition} \begin{proof} See Iwaniec and Kowalski \cite{IK} section 5.2.\end{proof}  Applying this we obtain \est{\sum_{f\in \mathscr{F}_\kappa(q) }\omega_f \Lambda(1/2+\alpha_1,f\otimes \chi) \Lambda(1/2+\alpha_2,f\otimes \chi)\Lambda(1/2+\alpha_3,f\otimes \chi) = \sum_{\sigma \in(\Z/2\Z)^3} \Delta(\mathscr{F},\sigma\alpha),} where \est{\Delta(\mathscr{F},\alpha) \eqdef  \prod_{i=1}^3 \frac{\chi(n_i)}{n_i^{1/2}} V_{1/2+\alpha_i}(n_i/q) \sum_{f\in \mathscr{F}_\kappa(q)} \omega_f \lambda_f(n_1)\lambda_f(n_2)\lambda_f(n_3) .} We work with a single $\Delta(\mathscr{F},\alpha)$ and leave the sum over $(\Z/2\Z)^3$ to the end.

Next we apply the Petersson trace formula to $\Delta(\mathscr{F},\alpha)$.  Following Conrey and Iwaniec formulae (2.9) and (2.11) we set \est{J(x) \eqdef 4 \pi i^\kappa x^{-1} J_{\kappa-1}(2 \pi x)} with $J_\nu(y)$ the standard $J$-Bessel function of the first kind.  The Petersson formula is \es{\label{Petersson}\sum_{f\in \mathscr{F}_\kappa (q)} \omega_f \lambda_f(m)\lambda_f(n) = \delta_{m=n} + \sqrt{mn} \sum_{c \equiv 0\pmod q} \frac{S(m,n,c)}{c^2}J(2 \sqrt{mn}/c),} where \est{S(m,n,c) \eqdef \sum_{\substack{a \pmod c \\ (a,c)=1}} e\left(\frac{am+\overline{a}n}{c}\right)} is the standard Kloosterman sum.  Let $e_c(x) = e(x/c) = e^{2 \pi i x/c}$ and define \est{V(x_1,x_2,x_3) \eqdef  V_{1/2+\alpha_1}(x_1) \sum_{(d,q)=1} \frac{1}{d}V_{1/2+\alpha_2}(dx_2)V_{1/2+\alpha_3}(dx_3).} Using Hecke multiplicativity and the Petersson formula \eqref{Petersson}, we find that \es{\label{offD}\Delta(\mathscr{F},\alpha) = \mathscr{D}_\alpha+\sum_{c \equiv 0(q)} \frac{\mathscr{S}_\alpha(c)}{c^2},} where the diagonal is given by \est{\mathscr{D}_\alpha \eqdef  \sum_{(n_1,q)=1}\frac{1}{n_1}  \sum_{n_1=n_2n_3} V\left(\frac{n_1}{q},\frac{n_2}{q},\frac{n_3}{q}\right) } and the off-diagonal is given by \est{\mathscr{S}_\alpha(c) \eqdef  \sum_{(n_1,n_2,n_3) \in \N^3} \chi(n_1n_2n_3) S(n_1,n_2n_3,c) J\left(2 \frac{\sqrt{n_1n_2n_3}}{c}\right) V\left(\frac{n_1}{q},\frac{n_2}{q},\frac{n_3}{q}\right) .}  

We now apply Poisson summation 3 times to $\mathscr{S}_\alpha(c)$ and change variables to find that \es{\label{S}\mathscr{S}_\alpha(c) = \sum_{(m_1,m_2,m_3)\in\Z^3}G(m_1,m_2,m_3,c)\check{W}_\alpha(m_1,m_2,m_3,c),} where following the notation of Conrey and Iwaniec, we have defined \est{G(m_1,m_2,m_3,c) \eqdef \sum_{(a_1,a_2,a_3) \in \left(\Z/c\Z\right)^3} \chi(a_1a_2a_3)S(a_1,a_2a_3,c)e_c(m_1a_1+m_2a_2+m_3a_3)} and \es{\label{Wdef}\check{W}_\alpha(m_1,m_2,m_3,c) \\ \eqdef \iiint_{\R^3_{>0}} J(2 \sqrt{cx_1x_2x_3})V\left(\frac{cx_1}{q},\frac{cx_2}{q},\frac{cx_3}{q}\right)e(-m_1x_1-m_2x_2-m_3x_3)dx_1\,dx_2\,dx_3 .}  
The $G(m_1,m_2,m_3,c)$ here is identical to that of Conrey and Iwaniec studied in sections 10,11,13 and 14 of their paper, and is independent of $\alpha$. We study $\check{W}_\alpha$ extensively in section \ref{sec:statphase}.  The formula \eqref{S} gives a decomposition of $\mathscr{S}_\alpha(c)$ into archimedian and non-archimedian parts, that is to say, $G$ is purely arithmetic and $\check{W}_\alpha$ is purely analytic. 

\section{The Main Terms}\label{Diagonal}

In this section we prove that 

\est{\sum_{\sigma \in (\Z/2\Z)^3} \left(\mathscr{D}_{\sigma\alpha} + \sum_{c\equiv 0 \pmod q} \frac{1}{c^2} \sum_{m_1m_2m_3=0} G(m_1,m_2,m_3,c) \check{W}_{\sigma \alpha} (m_1,m_2,m_3,c)\right) \\ = \sum_{\sigma \in (\Z/2\Z)^3} MT(\mathscr{F},\sigma \alpha) + O_\eps(q^{-1/3+\eps}).}
Set \es{\label{L}L(u_1,u_2,u_3) \eqdef  (u_1+u_2)(u_2+u_3)(u_3+u_1)MT(\mathscr{F},u).} The function $L$ is holomorphic, symmetric, and rapidly decaying in vertical strips in the region \est{\{(u_1,u_2,u_3)\in \C^3| \real(u_i)>-\kappa/2 \text{ for } i=1,2,3\}.} One has \begin{small}\es{\label{D}\mathscr{D}_\alpha=\frac{1}{(2\pi i )^3} \int_{(1/3)}\int_{(5/12)}\int_{(1/2)} \frac{L(u_1,u_2,u_3)}{(u_1-\alpha_1)(u_2-\alpha_2)(u_3-\alpha_3)(u_1+u_2)(u_2+u_3)(u_3+u_1)}\,du_3\,du_2\,du_1.}\end{small}  
Shifting the contours in \eqref{D} produces terms of the form \es{\label{H} M(\alpha,\beta,\gamma) \eqdef \frac{1}{2\pi i}\int_{(1/3)} \frac{L(u,-u,\gamma)}{(u-\alpha)(-u-\beta)(\gamma-u)(\gamma+u)}\,du.}  The function $L(u,-u,\gamma)$ is at least constant-sized on any vertical strip, so we cannot resolve $M(\alpha, \beta, \gamma)$ by contour shifting.  However, if we set \est{ N(\alpha,\beta,\gamma) \eqdef M(\alpha,\beta,\gamma) + \frac{L(\beta,-\beta, \gamma)}{(-\beta-\alpha)(\gamma + \beta)(\gamma -\beta)} + \frac{L(\gamma,-\gamma,\gamma)}{(\gamma-\alpha)(-\gamma-\beta)(2\gamma)}} then shifting contours gives \es{\label{Nshift}N(\alpha,\beta,\gamma) = N(\beta,\alpha,\gamma).} 
An intricate but elementary contour shift calculation shows that \es{\label{DiagFinal} \mathscr{D}_\alpha = MT(\mathscr{F},\alpha) + N(\alpha_2,\alpha_3,\alpha_1) + N(\alpha_1,\alpha_3,\alpha_2) + N(\alpha_1,\alpha_2,\alpha_3)  + \frac{1}{2}\frac{L(0,0,0)}{(-\alpha_1)(-\alpha_2)(-\alpha_3)} \\  + O_\eps(q^{-1/3+\eps}).}  
The resulting asymptotic formula \eqref{DiagFinal} for $\mathscr{D}_\alpha$ is symmetric in $\alpha_1,\alpha_2,\alpha_3$ due to \eqref{Nshift}. The terms of \eqref{DiagFinal} given by $N$ and $L(0,0,0)$ do not appear in the final answer predicted by the conjectures of the five authors \cite{CFKRS}.  Some of these terms cancel out after introducing the sum over $(\Z/2\Z)^3$, and others will be cancelled by off-diagonal main terms.

Our next goal in this section is to calculate the contribution of those terms of the off-diagonal (see \eqref{offD}) \est{\sum_{c \equiv 0 \pmod q} \frac{\mathscr{S}_\alpha(c)}{c^2} = \sum_{c \equiv 0 \pmod q}\frac{1}{c^2}\sum_{(m_1,m_2,m_3)\in\Z^3}G(m_1,m_2,m_3,c)\check{W}_\alpha(m_1,m_2,m_3,c)} whose indices satisfy $m_1m_2m_3=0$, i.e. which lie on one of the coordinate planes in $\Z^3$.  

We use the calculation of the arithmetic sum $G$ from Conrey and Iwaniec.  For the full statement of their calculation of $G$, see Lemma \ref{Gcalc} in section \ref{Dual}.  In this section we record only the following special cases.  Let $R_k(m)=S(0,m,k)$ denote the Ramanujan sum, and assume $m_i\neq 0$ for $i=1,2,3$.

\es{\label{G000}G(0,0,0,rq) = \begin{cases} \chi(-1) q \phi(q)^2 & r=1 \\ 0 & r>1,\end{cases}}
\es{\label{Gm00}G(m_1,0,0,rq) =   \chi(-1) r^2 q  \phi(q) R_q(m_1) \mathds{1}_{(m_1q,r)=1}, }
\es{\label{G0m0}G(0,m_2,0,rq) = \begin{cases} \chi(-1) q \phi(q) R_q(m_2) & r=1 \\ 0 & r>1,\end{cases}} and symmetrically for $G(0,0,m_3,rq)$,
\es{\label{G0mm}G(0,m_2,m_3,rq) = \begin{cases} \chi(-1) q R_q(m_2)R_q(m_3) & r=1 \\ 0 & r>1, \end{cases}}
\es{\label{Gmm0}G(m_1,m_2,0,rq) = \chi(-1) r^2 q R_q(m_1)R_q(m_2)\mathds{1}_{(m_1q,r)=1},} and symmetrically $ G(m_1,0,m_3,rq)$.

To calculate the analytic part, we use the following easily established Mellin transforms.
\es{\label{MellinJ}J(2\sqrt{c x_1x_2x_3}) =  \frac{\chi(-1)}{2\pi i } \int_{(3/4)} (2\pi)^{2s} \frac{\Gamma(\kappa/2-s)}{\Gamma(\kappa/2 +s) }(cx_1x_2x_3)^{s-1}\,ds} which is valid \`a priori on vertical lines $1/4< \real(s)< \kappa/2$, see for example formula 17.43.16 of Gradshteyn and Ryzhik \cite{GR7}.  
\est{e(-mx) = \frac{1}{2\pi i} \int_{(1/2)} \frac{\Gamma(s)}{(2\pi i m)^s} x^{-s}\,ds} valid \`a priori on vertical lines $0<\real(s) <1$, see formulas 17.43.3 and 17.43.4 of \cite{GR7}.  We have 
\est{V_{1/2+\alpha}(nx) = \frac{1}{2\pi i} \int_{(2)} \frac{\Gamma(\kappa/2+u)}{(u-\alpha)} (2\pi n)^{-u}x^{-u}\,du} for $n\in \N$ by definition, and 
\es{\label{Ve}\int_0^\infty V_{1/2+\alpha}(nx)  e(-m x) x^{s}\,\frac{dx}{x} = \frac{1}{2\pi i }\int_{(1/4)} \frac{\Gamma(\kappa/2+u)}{(u-\alpha)} (2\pi n)^{-u}\frac{\Gamma(s-u)}{(2\pi i m)^{s-u}}\,du} by Mellin convolution.  Shifting the contour sufficiently far to the left, such a formula is valid for any $\real(s)>\real(\alpha)$. One can rigorously justify the interchange of integrations by splitting the $x$-integral in two and applying integration by parts (i.e. the below Lemma \ref{BKYlemma}) to the tail.  

Given these formulae, one easily establishes the following Mellin inversion formulas for the integral $\check{W}_\alpha(m_1,m_2,m_3,c)$ assuming $m_i\neq 0$ for $i=1,2,3$.  

\es{\label{W000}\check{W}_\alpha(0,0,0,q) = \frac{\chi(-1)}{q} \frac{1}{2\pi i } \int_{(3/4)} \frac{\Gamma(\kappa/2-s)\Gamma(\kappa/2+s)^2(q/2\pi)^s}{(s-\alpha_1)(s-\alpha_2)(s-\alpha_3)}  \zeta_q(1+2s) \,ds,}
\es{\label{Wm00}\check{W}_\alpha(m_1,0,0,rq) = \frac{\chi(-1)}{rq} \frac{1}{2\pi i}\int_{(3/4)} \frac{\Gamma(\kappa/2-s)\Gamma(\kappa/2+s)}{(s-\alpha_2)(s-\alpha_3)} \zeta_q(1+2s)\left(\frac{q}{r|m_1|}\right)^s \\ \times \frac{1}{2\pi i } \int_{(1/4)} \frac{\Gamma(\kappa/2 + u)}{(u-\alpha_1)} \left(\frac{|m_1|}{2\pi r}\right)^u \frac{\Gamma(s-u)}{(2\pi i \sgn m_1)^{s-u}}\,du\,ds,}
\es{\label{W0m0}\check{W}_\alpha(0,m_2,0,q) = \frac{\chi(-1)}{q} \frac{1}{2\pi i}\int_{(3/4)} \frac{\Gamma(\kappa/2-s)\Gamma(\kappa/2+s)}{(s-\alpha_1)(s-\alpha_3)}\left(\frac{q}{|m_2|}\right)^s \\ \times \frac{1}{2\pi i } \int_{(1/4)} \frac{\Gamma(\kappa/2 + u)}{(u-\alpha_2)} \left(\frac{|m_2|}{2\pi }\right)^u \frac{\Gamma(s-u)}{(2\pi i \sgn m_2)^{s-u}}\,du\,ds,} and similarly if the roles of $m_2$ and $m_3$ are reversed,
\es{\label{W0mm}\check{W}_\alpha(0,m_2,m_3,q) = \frac{\chi(-1)}{q} \frac{1}{2\pi i}\int_{(3/4)}  \frac{\Gamma(\kappa/2 -s)}{(s-\alpha_1)} \left(\frac{2 \pi q}{|m_2m_3|}\right)^s \frac{1}{(2\pi i)^2} \iint_{(1/4)} \prod_{i=2}^3 \frac{\Gamma(u_i+\kappa/2)}{(u_i-\alpha_i)} \\ \times \left(\frac{|m_i|}{2\pi } \right)^{u_i} \frac{\Gamma(s-u_i)}{(2\pi i \sgn m_i)^{s-u_i}} \zeta_q(1+u_2+u_3)\,du_i\,ds,}
\es{\label{Wmm0}\check{W}_\alpha(m_1,m_2,0,rq) = \frac{\chi(-1)}{rq} \frac{1}{2\pi i}\int_{(3/4)}  \frac{\Gamma(\kappa/2 -s )}{(s-\alpha_3)} \left(\frac{2 \pi q}{|m_1m_2|}\right)^s \frac{1}{(2\pi i )^2} \iint_{(1/4)} \prod_{i=1}^{2}  \frac{\Gamma(u_i+\kappa/2)}{(u_i-\alpha_i)} \\ \times \left(\frac{|m_i|}{2\pi r} \right)^{u_i} \frac{\Gamma(s-u_i)}{(2\pi i \sgn m_i)^{s-u_i}} \zeta_q(1+s+u_1)\,du_i\,ds,} and similarly if the role of $m_2$ is played by $m_3$ instead.  Note in particular that all of holomorphic functions appearing above are rapidly decaying in the appropriate vertical strips, so we are free to use contour shift arguments in the following.

Recall the definition of $L(u_1,u_2,u_3)$ from \eqref{L}.

\begin{lemma}\label{000}
We have as $q\rightarrow \infty$ that \est{ & \sum_{r=1}^\infty\frac{1}{r^2q^2}G(0,0,0,rq)\check{W}_\alpha(0,0,0,rq) \\ =&  \frac{L(\alpha_1,\alpha_1,-\alpha_1)}{(\alpha_1-\alpha_2)(\alpha_1-\alpha_3)(2 \alpha_1)} + \frac{L(\alpha_2,\alpha_2,-\alpha_2)}{(\alpha_2-\alpha_1)(\alpha_2-\alpha_3)(2\alpha_2)}  +\frac{L(\alpha_3,\alpha_3,-\alpha_3)}{(\alpha_3-\alpha_1)(\alpha_3-\alpha_2)(2 \alpha_3)} \\ &+ \frac{1}{2} \frac{L(0,0,0)}{(-\alpha_1)(-\alpha_2)(-\alpha_3)}+ O(q^{-1/2}).}
\end{lemma}
\begin{proof} Follows directly from \eqref{G000}, \eqref{W000}, and a contour shift.
\end{proof}

Now recall the definition of $M(u_1,u_2,u_3)$ from \eqref{H}.

\begin{lemma}\label{m2m3}
We have as $q \rightarrow \infty$ that \est{\sum_{r=1}^\infty \sum_{m_2 \neq 0 } \frac{G(0,m_2, 0,rq)}{r^2q^2} \check{W}_\alpha(0,m_2,0,rq) = -M(\alpha_1,-\alpha_3,\alpha_2)+ O(q^{-1/2}),} and similarly \est{\sum_{r=1}^\infty \sum_{m_3 \neq 0 } \frac{G(0,0, m_3,rq)}{r^2q^2} \check{W}_\alpha(0,0,m_3,rq) = -M(\alpha_1,-\alpha_2,\alpha_3)+ O(q^{-1/2}).}
\end{lemma}
\begin{proof}
We prove only the first formula as the second is identical after swapping $m_2$ and $m_3$.  Putting \eqref{G0m0} and \eqref{W0m0} together we are led to consider the Dirichlet series associated with the Ramanujan sums $R_q(m)$.  We have the nice formula \es{\label{nice formula} \sum_{m\geq 1} \frac{R_q(m)}{m^s} = \zeta(s) q^{1-s} \prod_{p \mid q} \left(1-p^{-(1-s)}\right),}  and in fact we have \est{  \frac{\Gamma(s-u)}{(2\pi i )^{s-u} }\sum_{m_2\geq 1} \frac{R_q(m_2)}{m_2^{s-u}} + \frac{\Gamma(s-u)}{(-2\pi i )^{s-u} }\sum_{m_2\leq -1} \frac{R_q(m_2)}{|m_2|^{s-u}} = \zeta_q(1-s+u)q^{1-(s-u)} } using the asymmetric functional equation.  Then we have that \est{q^{-2}\sum_{m_2 \neq 0 } G(0,m_2, 0,q) \check{W}_\alpha(0,m_2,0,q) = & \frac{\phi(q)}{q} \frac{1}{2\pi i } \int_{(3/4)} \frac{\Gamma(\kappa/2-s)\Gamma(\kappa/2+s)}{(s-\alpha_1)(s-\alpha_3)}  \frac{1}{2\pi i } \int_{(1/4)} \frac{\Gamma(k/2+u)}{(u-\alpha_2)} \\ & \times \left(\frac{q}{2\pi}\right)^u \zeta_q(1+s+u)\zeta_q(1-s+u)\,du\,ds.}  We may now shift the $u$ integral to $\real(u) =-1/2$ and pick pick up the residue at $u=\alpha_2$ to conclude the lemma.
\end{proof}

 \begin{lemma}\label{m1}
We have as $q \rightarrow \infty$ that \est{\sum_{r=1}^\infty \frac{1}{r^2q^2}\sum_{m_1 \neq 0} G(m_1,0,0,rq) \check{W}_\alpha(m_1, 0,0,rq) = -M(\alpha_2,-\alpha_3,\alpha_1)+ O(q^{-1/2}).}
\end{lemma}
\begin{proof} In similar fashion to the proof of the previous lemma we combine formulas \eqref{Gm00} and \eqref{Wm00}.  The sum over $m_1$ leads us to use the formula \est{\sum_{\substack{m_1\neq 0 \\ (m_1,r)=1}} \frac{R_q(m_1)}{|m_1|^{s-u}} \frac{\Gamma(s-u)}{(2\pi i \sgn m_i)^{s-u} }  =  \zeta_q(1-s+u)\prod_{p \mid r} \left(1-p^{-(s-u)}\right) q^{1-(s-u)} .} We next use the formula \est{ \sum_{\substack{r\geq 1 \\ (r,q)=1}} \frac{1}{r^{1+s+u}} \prod_{p \mid r} \left(1-p^{-(s-u)}\right) = \frac{\zeta_q(1+s+u)}{\zeta_q(1+2s)}} to evaluate the sum over $r$.  Assembling these pieces we have that \est{\sum_{r=1}^\infty \frac{1}{r^2q^2}\sum_{m_1 \neq 0} G(m_1,0,0,rq) \check{W}_\alpha(m_1, 0,0,rq) =  \frac{\phi(q)}{q} \frac{1}{2\pi i } \int_{(3/4)} \frac{\Gamma(\kappa/2-s)\Gamma(\kappa/2+s)}{(s-\alpha_2)(s-\alpha_3)} \\ \times \frac{1}{2\pi i } \int_{(1/4)} \frac{\Gamma(k/2+u)}{(u-\alpha_1)}  \left(\frac{q}{2\pi}\right)^u \zeta_q(1+s+u)\zeta_q(1-s+u)\,du\,ds.}  We may now shift the $u$ integral to $\real(u) =-1/2$ and pick pick up the residue at $u=\alpha_1$ to conclude the lemma.
\end{proof}

We have now found all of the main terms for this cubic moment which are predicted by the conjectures of Conrey et al \cite{CFKRS}.  Re-introducing the sum over $(\Z/2\Z)^3$ from section \ref{Initial} the reader will observe the cancellation of many terms from \eqref{DiagFinal} and Lemmas \ref{000}, \ref{m2m3} and \ref{m1}.

Although not contributing to the main terms, observe the terms which come from $(m_1,m_2,m_3)$ lying on a coordinate plane but not on a coordinate axis also are also easily estimable by the techniques of this section.  For example combining equations \eqref{G0mm}, \eqref{W0mm} and formulas similar to \eqref{nice formula} one is led to \est{& q^{-2}\sum_{\substack{m_2 \neq 0 \\ m_3 \neq 0}} G(0,m_2,m_3,q)\check{W}_\alpha(0,m_2,m_3,q) \\ = & \frac{1}{2 \pi i  } \int_{(3/4)} (2 \pi)^s \frac{\Gamma(\kappa/2 -s)}{(s-\alpha_1)} q^{-s} \frac{1}{(2\pi i )^2} \iint_{(1/8)} \prod_{i=2}^3 \frac{\Gamma(\kappa/2 +u_i)}{(u_i-\alpha_i)} \left(\frac{q}{2\pi}\right)^{u_i} \\ & \times \zeta_q(1-s+u_i)\zeta_q(1+u_2+u_3)\,du_i\,ds\ll q^{-1/2}.}  One may treat in exactly the same fashion the sums arising from the terms \eqref{Gmm0} and \eqref{Wmm0} and likewise find that these are $\ll q^{-1/2}$.

\section{An Integral}\label{sec:statphase}

In this section we give Mellin formulas for the $m_1m_2m_3 \neq 0$ case of the integral $\check{W}_\alpha(m_1,m_2,m_3,c)$ (cf. Lemma 8.1 of \cite{ConreyIwaniec} or \cite{YoungCubic}). Recall the definition of $\check{W}_\alpha:$ \est{\check{W}_\alpha= \iiint_{\R^3}J(2\sqrt{cx_1x_2x_3}) V\left(\frac{cx_1}{q},\frac{cx_2}{q},\frac{cx_3}{q}\right)e(-m_1x_1-m_2x_2-m_3x_3)\,dx_1\,dx_2\,dx_3.}  
For large arguments the function $J(x)$ oscillates with unit period.  We run an elaborate stationary phase argument on the Bessel function and the complex exponentials above to derive a Mellin formula for $\check{W}_\alpha$.  In the cases where no stationary point of the phase exists we are able to compute the Mellin transforms directly.  Our stationary phase argument is based on section 8 of a paper of Blomer, Khan and Young \cite{BKY} but adapted to three variables using ideas from Stein \cite{SteinHarmonicAnalysis} chapter VIII.  

\begin{proposition}\label{Oscillatory}
Suppose that either all $m_i>0$ or all $m_i<0$. We have that \es{\label{oscillatoryformula}\check{W}_{\alpha}(m_1,m_2,m_3,c) = T_1 + T_2 + O\left(\min\left(q^{-2014}, \frac{q^{12}}{c^9}\frac{1}{|m_1m_2m_3|^2}\right) \right)}where \est{T_1 = & e\left(-\frac{m_1m_2m_3}{c}\right)\frac{\chi(-1)}{c} \frac{1}{(2\pi i)^4} \int_{(3/4)}\iiint_{(1/16)} \frac{U_1(s,u_1,u_2,u_3)\zeta_q(1+u_2+u_3)}{(u_1-\alpha_1)(u_2-\alpha_2)(u_3-\alpha_3)}   \\ & \times \left(\frac{c}{|m_1m_2m_3|}\right)^s \left(\frac{|m_2m_3|}{q}\right)^{-u_1}\left(\frac{|m_1m_3|}{q}\right)^{-u_2}\left(\frac{|m_1m_2|}{q}\right)^{-u_3}\,du_1\,du_2\,du_3\,ds} for a holomorphic function $U_1$ rapidly decaying in the vertical strips \est{ -\frac{1}{2} -\real(\sum u_i) < & \real(s) <1 \\ -\frac{\kappa}{2}< & \real(u_i).}  Similarly \est{T_2 = & \frac{\chi(-1)}{c} \frac{1}{(2\pi i)^4} \int_{(3/4)}\iiint_{(1/16)} \frac{U_2(s,u_1,u_2,u_3)\zeta_q(1+u_2+u_3)}{(u_1-\alpha_1)(u_2-\alpha_2)(u_3-\alpha_3)}   \\ &  \times \left(\frac{c}{|m_1m_2m_3|}\right)^s \left(\frac{|m_2m_3|}{q}\right)^{-u_1}\left(\frac{|m_1m_3|}{q}\right)^{-u_2}\left(\frac{|m_1m_2|}{q}\right)^{-u_3}\,du_1\,du_2\,du_3\,ds} for a holomorphic function $U_2$ rapidly decaying in the vertical strips \es{\label{regionvalid}\max_{i\neq j} \left( -\real(u_i+u_j) , -\real(\sum u_i)\right) < & \real(s) <\frac{\kappa}{2} -\real(\sum u_i) \\ -\frac{\kappa}{2}< & \real(u_i).} Moreover, $U_2$ has a meromorphic continuation with a simple polar divisor at $s+u_1+u_2+u_3 = \kappa/2$.  

Suppose now that the $m_i$ are of mixed signs.  We have \est{\check{W}_\alpha(m_1,m_2,m_3,c) = T_3(m_1,m_2,m_3,c)} where $T_3$ has the same definition as $T_2$, but where the holomorphic function also depends on the signs of the $m_i$.  Otherwise each of these holomorphic functions have exactly the same above properties as $U_2$.
\end{proposition}

\begin{proof}
We consider first the most difficult case when all $m_i>0$ or all $m_i<0$.  Computing directly the Mellin transform of $\check{W}_\alpha$ in these cases gives a function which does not decay rapidly in vertical strips.  We instead apply the method of stationary phase.  We use the formula  \es{\label{Jintegralformula}J_{\kappa-1}(x) = & \frac{1}{\pi} \int_{0}^\pi \cos(x \sin \theta - (\kappa-1)\theta)\,d\theta,} valid for integral $\kappa -1$, from which we find \est{J(x) = 2\chi(-1)\sum_{\pm} \int_0^\pi e(\mp(\kappa-1) \theta/2\pi) \frac{e(\pm x \sin \theta)}{x}\,d\theta .}  We insert this formula for $J$ in the definition of $\check{W}_\alpha$ and pull the integrals over $\theta$ outside.  Let $w_1(t)$ denote a smooth function on $\R_{>0}$ which is identically $0$ for $t \leq 1/2$ and identically 1 for $t\geq 1$. Set $c_\theta = c \sin^2 \theta$. Split the 3 defining integrals of $\check{W}_\alpha$ at $c_\theta x_1x_2x_3 =1$ using the function $w_1$.  Below the hyperboloid $c_\theta x_1x_2x_3 =1$ we move the $\theta$ integrals back inside.
On a first reading the reader will lose no essential details by taking $\theta=\pi/2$ so that $c_\theta=c$, as this is the ``true'' phase of the Bessel function for large $x$ anyhow. 
These maneuvers result in the decomposition: \es{\label{decomposition1} \check{W}_\alpha = \check{R}_\alpha + 2i^\kappa \int_0^\pi e(-(\kappa-1) \theta/2\pi) (\sin \theta)  \check{U}^+_\alpha(\theta)\,d\theta + 2i^\kappa \int_0^\pi  e((\kappa-1) \theta/2\pi) (\sin \theta) \check{U}^-_\alpha(\theta)\,d\theta,} where: \begin{itemize} 
\item The term $\check{R}_\alpha$ represents the integral under the hyperbola, i.e. \est{\check{R}_\alpha \eqdef \iiint_{\R^3}R(2\sqrt{cx_1x_2x_3}) V\left(\frac{cx_1}{q},\frac{cx_2}{q},\frac{cx_3}{q}\right)e(-m_1x_1-m_2x_2-m_3x_3)\,dx_1\,dx_2\,dx_3} with \est{R(x) \eqdef \frac{4\chi(-1)}{x} \int_0^\pi \cos(2 \pi x \sin \theta - (\kappa-1)\theta)\left(1-w_1\left((x\sin \theta)^2/4\right)\right)\,d\theta.} The function $R(x)$ is identically equal to $J(x)$ when $x<\sqrt{2}$, satisfies the bound $x^n R^{(n)}(x) \ll_n 1$ for all $n\in \N$, and is $\ll_\kappa x^{-2}$ for large $x$. 
\item The $\check{U}^\pm_\alpha(\theta)$ are oscillatory integrals to which we apply the method of stationary phase.  Specifically, define the amplitude \es{\label{RVdef} V(\mathbf{x}) \eqdef V\left(\frac{cx_1}{q},\frac{cx_2}{q},\frac{cx_3}{q}\right)\frac{w_1(c_\theta x_1x_2x_3)}{2\sqrt{c_\theta x_1x_2x_3}}} and the phase \est{f^\pm(\mathbf{x}) \eqdef \pm2\sqrt{c_\theta x_1x_2x_3} -m_1x_1-m_2x_2-m_3x_3.}  Then the oscillatory integrals are given by \est{\check{U}^\pm_\alpha \eqdef \iiint_{\R_{>0}^3}V(\mathbf{x})e(f^\pm(\mathbf{x}))\,d\mathbf{x}.} The amplitude function is non-oscillatory, i.e. it satisfies \est{ x_1^{n_1}x_2^{n_2}x_3^{n_3}\partial^{(n_1,n_2,n_3)}V(\mathbf{x})\ll_{n_1,n_2,n_3,A} \left(1+\frac{c|x_1|}{q}\right)^{-A} \left(1+\frac{c|x_2|}{q}\right)^{-A} \left(1+\frac{c|x_3|}{q}\right)^{-A}.}
\end{itemize}
We will be able to give the Mellin transform of $\check{R}_\alpha$ directly so we instead first focus on the oscillatory integrals $\check{U}^\pm_\alpha$.   
A critical point of $f^+(\mathbf{x})$ exists in the positive octant of $\mathbf{x} \in \R^3$ if and only if all $m_i>0$.  Similarly a critical point of $f^-(\mathbf{x})$ exists in the positive octant of $\mathbf{x} \in \R^3$ if and only if all $m_i<0$.  We focus on the $\check{U}^+_\alpha$ and ``all $m_i>0$'' case until further notice, the ``all $m_i<0$'' case being treated similarly.  In the $\check{U}^+_\alpha$ case the unique critical point occurs at $x_0= \left(\frac{m_2m_3}{c_\theta}, \frac{m_1m_3}{c_\theta},\frac{m_1m_2}{c_\theta}\right)$ and we apply the method of stationary phase to $\check{U}^+_\alpha$ at $x_0$.  

Before all else, if any of the $m_i$ or $c$ are extremely large relative to $q$, we may integrate $\check{U}^\pm_\alpha$ by parts several times to obtain the error term in \eqref{oscillatoryformula}.  This allows us to keep track of only the $q$ dependence in the error terms in the following application of the stationary phase method.

We recall an extremely useful Lemma of Blomer, Khan and Young from \cite{BKY}.  The essence of this Lemma is ``integration by parts''.

\begin{lemma}[Blomer-Khan-Young] \label{BKYlemma}
 Let $Y \geq 1$, $X, Q, U, R > 0$,  
and suppose that $w$ 
 is a smooth function with support on $[\alpha, \beta]$, satisfying
\begin{equation*}
w^{(j)}(t) \ll_j X U^{-j}.
\end{equation*}
Suppose $h$ 
  is a smooth function on $[\alpha, \beta]$ such that
\est{ |h'(t)| \geq R}
for some $R > 0$, and
\est{h^{(j)}(t) \ll_j Y Q^{-j}, \qquad \text{for } j=2, 3, \dots.}
Then the integral $I$ defined by
\begin{equation*}
I = \int_{-\infty}^{\infty} w(t) e^{i h(t)} dt
\end{equation*}
satisfies
\est{ I \ll_A (\beta - \alpha) X [(QR/\sqrt{Y})^{-A} + (RU)^{-A}],} where the implied constants depend only on $A$. 
\end{lemma}

Apply a smooth dyadic partition of unity to the integral $\check{U}^+_\alpha$ which localizes the variable $x_i$ at $X_i$.  In Lemma \ref{BKYlemma} we take \est{X = & 1, \\ U= & X_i,\\ R=& X_i^{-1}(\sqrt{c_\theta X_1X_2X_3}-m_iX_i) \\ Y= & \sqrt{c_\theta X_1X_2X_3} \\ Q= & X_i.}  Let $\eps_1, \eps_2>0$ be small real numbers.  Lemma \ref{BKYlemma} says that the integral $\check{U}^+_\alpha$ on the $X_1,X_2,X_3$-piece of the partition is extremely small if any of the following three hold: \es{\label{starfirst} \left| \sqrt{c_\theta X_1X_2X_3} -m_1X_1\right|  >  q^{\eps_1} \text{ and } \left| \sqrt{c_\theta X_1X_2X_3} -m_1X_1\right|   >  q^{\eps_2} (c_\theta X_1X_2X_3)^{1/4}\\  \left| \sqrt{c_\theta X_1X_2X_3} -m_2X_2\right|  >  q^{\eps_1} \text{ and } \left| \sqrt{c_\theta X_1X_2X_3} -m_2X_2\right|   >  q^{\eps_2} (c_\theta X_1X_2X_3)^{1/4}\\ \left| \sqrt{c_\theta X_1X_2X_3} -m_3X_3\right|  >  q^{\eps_1} \text{ and } \left| \sqrt{c_\theta X_1X_2X_3} -m_3X_3\right|   >  q^{\eps_2} (c_\theta X_1X_2X_3)^{1/4}. }  We study the complimentary set to these inequalities where Lemma \ref{BKYlemma} is not applicable.  The description of this set and the behavior of our integral naturally breaks into two cases depending on the size of $m_1m_2m_3/c_\theta$.

\begin{lemma}[Localization 1]\label{localizationlemma1}
Suppose that $m_1m_2m_3/c_\theta > q^\delta $ for some small $\delta >0$.  Let $R_1 \subset \R_{>0}^3$ be the region surrounding the point \est{x_0 = \left(\frac{m_2m_3}{c_\theta},\frac{m_1m_3}{c_\theta},\frac{m_1m_2}{c_\theta}\right)} cut out by the inequalities \est{\left| \frac{m_2m_3}{c_\theta} -x_1\right| \leq & \frac{q^{\delta/2}}{m_1}\sqrt{\frac{m_1m_2m_3}{c_\theta}}\\ \left| \frac{m_1m_3}{c_\theta} -x_2\right| \leq &  \frac{q^{\delta/2}}{m_2}\sqrt{\frac{m_1m_2m_3}{c_\theta}}\\ \left| \frac{m_1m_2}{c_\theta} -x_3\right| \leq & \frac{q^{\delta/2}}{m_3}\sqrt{\frac{m_1m_2m_3}{c_\theta}}.} 
Let $w_0(t)$ denote a smooth function identically $1$ if $|t|\leq 1/2$ and identically $0$ if $|t|\geq 1$ and \est{L_1 = \iiint_{\R^3} V(\mathbf{x}) e(f^+(\mathbf{x})) \prod_{i=1}^3 w_0\left(\frac{\left| \frac{m_1m_2m_3}{c_\theta}-m_ix_i\right|}{q^{\delta/2} \sqrt{m_1m_2m_3/c_\theta}}\right)\,d\mathbf{x}.}  We have for sufficiently large $q$ that \est{\check{U}^+_\alpha = L_1 + O_\delta(q^{-2014}).}
\end{lemma}

\begin{proof}
We apply Lemma \ref{BKYlemma} to the integral $\check{U}^+_\alpha$.  Let \est{A_{1,j} \eqdef \left| \sqrt{c_\theta x_1x_2x_3} -m_jx_j\right|   \leq &  q^{\eps_1}} and \est{ A_{2,j} \eqdef \left| \sqrt{c_\theta x_1x_2x_3} -m_jx_j\right|   \leq &  q^{\eps_2} (c_\theta x_1x_2x_3)^{1/4}} be regions in $\R^3_{>0}$.  The complement in $\R^3_{>0}$ of the region defined by \eqref{starfirst} is \es{\label{starwholething} \bigcup_{\substack{i_1=1,2 \\ i_2=1,2 \\ i_3=1,2 }} A_{i_1,1}\cap A_{i_2,2}\cap A_{i_3,3}} for any positive $\eps_1$ and $\eps_2$.  The integrand of $\check{U}^+_\alpha$ is supported within $c_\theta x_1x_2x_3>1/2$, so taking $\eps = \eps_2 = \eps_1 + \log 2 / \log q$ the region $A_{2,1}\cap A_{2,2} \cap A_{2,3}$ contains all of the other regions in \eqref{starwholething}.  Therefore it suffices to study the region \es{\label{star2} \{c_\theta x_1x_2x_3 >1/2\} \cap  A_{2,1}\cap A_{2,2} \cap A_{2,3}.}
We split $\R_{>0}^3$ into two disjoint cases: all $m_ix_i < 3(c_\theta x_1x_2x_3)^{1/4}q^{\eps}$ or at least one $m_ix_i \geq 3 (c_\theta x_1x_2x_3)^{1/4}q^{\eps}$.  The former case defines a region which lies completely under the hyperboloid $c_\theta x_1x_2x_3<1/2$ when $\eps \leq \delta/3 - \frac{\log 2^{1/12}/3}{\log q}$ so we assume this restriction on $\eps$ and focus on the case that at least one $m_ix_i \geq 3 (c_\theta x_1x_2x_3)^{1/4}q^{\eps}$.
By applying the triangle inequality to pairs $|m_ix_i-m_jx_j|$ and \eqref{star2} we have \es{\label{lowerbounds2}  \sqrt{c_\theta x_1x_2x_3} \geq & q^\eps(c_\theta x_1x_2x_3)^{1/4} \\ m_ix_i \geq & q^\eps(c_\theta x_1x_2x_3)^{1/4}  } for $i=1,2,3$.  
By multiplying together the inequalities \eqref{star2} we have \est{ \left|\frac{m_1m_2m_3}{c_\theta } - m_ix_i \right| \leq  q^\eps (c_\theta x_1x_2x_3)^{1/4} \left( 1+  \frac{m_1m_2x_1x_2}{c_\theta  x_1x_2x_3} + \frac{m_1m_3x_1x_3}{c_\theta x_1x_2x_3}+ \frac{m_2m_3x_2x_3}{c_\theta x_1x_2x_3} \right. \\ \left. +q^{\eps} \left(  \frac{m_1x_1}{(c_\theta  x_1 x_2 x_3)^{3/4}} + \frac{m_2x_2}{(c_\theta  x_1 x_2 x_3)^{3/4}}+ \frac{m_3x_3}{(c_\theta  x_1 x_2 x_3)^{3/4}}\right)+ q^{2\eps} \frac{1}{(c_\theta  x_1 x_2 x_3)^{1/2}} \right).} 
 The inequalities \eqref{star2} and \eqref{lowerbounds2} together imply \est{0 \leq \frac{m_ix_i}{\sqrt{c_\theta x_1x_2x_3}} \leq 2} for all $i=1,2,3$, and we obtain \est{\left| \frac{m_1m_2m_3}{c_\theta } - m_ix_i\right|  \leq 13 q^\eps (c_\theta  x_1 x_2 x_3)^{1/4} + 7q^{2\eps}.} Now we also have by multiplying the inequalities \eqref{star2} that \est{\sqrt{c_\theta x_1x_2x_3} \leq \frac{m_1m_2m_3}{c_\theta } + 12(c_\theta x_1x_2x_3)^{1/4}q^{\eps} + 7 q^{2\eps},}  so that \es{\label{loc123}\left| \frac{m_1m_2m_3}{c_\theta } - m_ix_i\right|  \leq 13 q^\eps \sqrt{\frac{m_1m_2m_3}{c_\theta }}  + (13(6+\sqrt{43})+7) q^{2\eps} . }  
This region lies strictly above the hyperboloid $c_\theta x_1x_2x_3 =1$ as soon as \est{\frac{m_1m_2m_3}{c_\theta } \left(\frac{m_1m_2m_3}{c_\theta } - 13 q^\eps \sqrt{\frac{m_1m_2m_3}{c_\theta }}  - (13(6+\sqrt{43})+7) q^{2\eps}\right)^3>1,} which must become true for sufficiently large $q$ with our previous restriction on $\eps$, and moreover also the region defined by \eqref{loc123} is contained within the one stated in the Lemma.

Thus we may pass to a sufficiently fine partition of unity and apply the stationary phase Lemma \ref{BKYlemma} to localize the integral $\check{U}_\alpha$ to the region $R_1$. \end{proof}

\begin{lemma}[Localization 2]\label{localizationlemma2}

Suppose that $m_1m_2m_3/c_\theta \leq q^{\delta_2} $ for some small $\delta_2 >0$.  Let $R_2 \subset \R_{>0}^3$ be the region surrounding the point \est{x_0 = \left(\frac{m_2m_3}{c_\theta},\frac{m_1m_3}{c_\theta},\frac{m_1m_2}{c_\theta}\right)} cut out by the inequalities \est{\left| \frac{m_2m_3}{c_\theta} -x_1\right| \leq & \frac{q^{\delta_2}}{m_1} \\ \left| \frac{m_1m_3}{c_\theta} -x_2\right| \leq &  \frac{q^{\delta_2}}{m_2} \\ \left| \frac{m_1m_2}{c_\theta} -x_3\right| \leq &  \frac{q^{\delta_2}}{m_3}.} 
Let $w_0(t)$ denote a smooth function identically $1$ if $|t|\leq 1/2$ and identically $0$ if $|t|\geq 1$ and let \est{ L_2 = \iiint_{\R^3} V(\mathbf{x}) e(f^+(\mathbf{x})) \prod_{i=1}^3w_0\left(\left| \frac{m_1m_2m_3}{c_\theta}-m_ix_i\right|q^{-\delta_2}\right)\,d\mathbf{x} .}  For sufficiently large $q$ we have that \est{\check{U}_\alpha = L_2 + O_{\delta_2}(q^{-2014}).} 
\end{lemma}

\begin{proof}
The proof is similar to that of Lemma \ref{localizationlemma1} so we omit it.
\end{proof}

Our next goal is to give an asymptotic formula for the integrals $L_i$ in Lemmas \ref{localizationlemma1} and \ref{localizationlemma2}.  We begin by making a change of variables.  Let $\mathcal{U}\subset \R^3$ be the eighth-space defined by \begin{small}\est{\left\{ (v_1,v_2,v_3)\in \R^3 \Bigg. \Bigg| v_1 > - \sqrt{ \left( v_2+ m_3 \sqrt{\frac{m_1}{c_\theta }} \right)\left( v_3+ m_2 \sqrt{\frac{m_1}{c_\theta }} \right)}, v_2 > - \sqrt{\frac{ m_1}{c_\theta } }m_3 , v_3 > -  \sqrt{\frac{ m_1}{c_\theta } } m_2) \right\} .}\end{small} Let the change-of-variables diffeomorphism $\varphi: \mathcal{U} \xrightarrow{\,\sim\,} \R^3_{>0}$ be defined by \est{ x_1 = & \frac{1}{m_1} \left(v_1 + \sqrt{\frac{x_2x_3 c_\theta }{m_1}}\right)^2 \\ x_2 = & \sqrt{\frac{m_1}{c_\theta }} \left( v_2+ m_3 \sqrt{\frac{m_1}{c_\theta }} \right) \\ x_3 = &  \sqrt{\frac{m_1}{c_\theta }} \left( v_3+ m_2 \sqrt{\frac{m_1}{c_\theta }} \right).}   
The change of variables $\varphi$ transforms the phase function $f^+(\mathbf{x})$ to \est{f^+(\varphi(\mathbf{v})) = -\frac{m_1m_2m_3}{c_\theta } -v_1^2+ v_2v_3,} and moves the critical point to the origin \est{\varphi(0) = x_0.}  This is the same change of variables as found in Conrey and Iwaniec Lemma 8.1.  Such a change of variables is guaranteed to exist (locally) by Morse's lemma, see Stein \cite{SteinHarmonicAnalysis} chapter VIII, section 2.3.2.  Geometrically, the box $R_1$ given by Lemma \ref{localizationlemma1} is, up to absolute constants, of size \est{ \frac{q^{\delta/2}}{m_1}\sqrt{\frac{m_1m_2m_3}{c_\theta }} \times \frac{q^{\delta/2}}{m_2}\sqrt{\frac{m_1m_2m_3}{c_\theta }} \times \frac{q^{\delta/2}}{m_3}\sqrt{\frac{m_1m_2m_3}{c_\theta }} } around the critical point $x_0$ in the $x$-space, and the box $R_2$ given by Lemma \ref{localizationlemma2} is up to absolute constants of size \est{ \frac{q^{\delta_2}}{m_1} \times \frac{q^{\delta_2}}{m_2} \times \frac{q^{\delta_2}}{m_3} } around the critical point $x_0$ in the $x$-space.  The regions $\varphi^{-1}(R_1)$ and $\varphi^{-1}(R_2)$ are also right-angle parallelepipeds in $v$-space.  The boxes $\varphi^{-1}(R_1)$ and $\varphi^{-1}(R_2)$ both surround the origin in $v$-space of are up to absolute constants of size \est{ q^{\delta/2}  \times q^{\delta/2} \sqrt{\frac{m_3}{m_2}} \times \ q^{\delta/2} \sqrt{\frac{m_2}{m_3}}, } and \est{ q^{\delta_2} \sqrt{\frac{c_\theta }{m_1m_2m_3}} \times \frac{q^{\delta_2}}{m_2} \sqrt{\frac{c_\theta }{m_1}}\times \frac{q^{\delta_2}}{m_3} \sqrt{\frac{c_\theta }{m_1}},} respectively (although the origin is no longer in the center of the box in the $v_1$ direction).  
The determinant of the Jacobian of the change of variables is \est{ (\det {\rm Jac } \,\varphi)(\mathbf{v}) = \frac{2}{c_\theta } \left( v_1 + \sqrt{\left( v_3+ m_2 \sqrt{\frac{m_1}{c_\theta }} \right) \left( v_2+ m_3 \sqrt{\frac{m_1}{c_\theta }} \right)} \right)= \frac{2\sqrt{m_1x_1}}{c_\theta }.  } In particular at the stationary point we have \est{ (\det {\rm Jac} \, \varphi)(0) = 2\sqrt{ \frac{m_1m_2m_3}{c_\theta ^3}}.} Recall the definition of $w_0$ from Lemma \ref{localizationlemma1} and $V(\mathbf{x})$ from \eqref{RVdef}, and set \es{\label{g1def} g_1(\mathbf{v}) \eqdef  V(\varphi(\mathbf{v})) |(\det {\rm Jac} \, \varphi)(\mathbf{v})|  w_0\left(C_1\frac{v_1}{q^{\delta/2}}\right)w_0\left(\frac{v_2}{q^{\delta/2}}\sqrt{\frac{m_2}{m_3 }}\right)w_0\left(\frac{v_3}{q^{\delta/2}}\sqrt{\frac{m_3}{m_2 }}\right)} and \est{g_2(\mathbf{v}) \eqdef  V(\varphi(\mathbf{v})) |(\det {\rm Jac} \, \varphi)(\mathbf{v})|  w_0\left(C_2\frac{v_1}{q^{\delta_2}}\sqrt{\frac{m_1m_2m_3}{c_\theta }}\right)w_0\left(\frac{v_2}{q^{\delta_2}}\sqrt{\frac{m_1m_2^2}{c_\theta }}\right)w_0\left(\frac{v_3}{q^{\delta_2}}\sqrt{\frac{m_1m_3^2}{c_\theta }}\right),} for some absolute constants $C_1$ and $C_2$.  The integrals $L_i$ for $i=1,2$ are transformed to \est{ L_i= e \left(-\frac{m_1m_2m_3}{c_\theta }\right) \iiint_{\R^3} g_i(\mathbf{v})  e(-v_1^2+v_2v_3)\,d\mathbf{v}+ O(q^{-2014}).} 

The support of $g_1(\mathbf{v})$ is the right angle parallelepiped given by the functions $w_0$ in \eqref{g1def}.  If $m_1m_2m_3/c_\theta >1/4$ the support of $g_2(\mathbf{v})$ is a right-angle parallelepiped but with possibly also one corner, edge or face truncated by the hyperboloid cut-off by function $V(\varphi(\mathbf{v}))$.  In particular, when $m_1m_2m_3/c_\theta >1/4$ the support of $g_2(\mathbf{v})$ at least contains the stationary point $x_0$ and $1/8$-th of the box $\varphi^{-1}(R_2)$ which points in the totally positive direction in $v$-space.  If $m_1m_2m_3/c_\theta \leq 1/4$ then the point $x_0$ lies outside the support of $g_2(\mathbf{v})$.  

By Fourier inversion we have for any small $\eps>0$ that\est{g_1(\mathbf{v}) =  \iiint_{\widehat{R}_i}  \widehat{g}_1(\mathbf{y}) e(\mathbf{v} \cdot \mathbf{y}) \,d\mathbf{y} + O_{\eps,\delta}(q^{-2014})} where \est{ \widehat{R}_i \eqdef \begin{cases} \substack{\begin{cases}|y_1| \ll  q^{\eps-\delta/2} \\ |y_2| \ll \sqrt{\frac{m_2}{m_3}} q^{\eps-\delta/2}  \\ |y_3| \ll \sqrt{\frac{m_3}{m_2}} q^{\eps-\delta/2} \end{cases} } & \text{ if } i = 1 \\ \substack{\begin{cases}|y_1| \ll  \sqrt{\frac{m_1m_2m_3}{c_\theta}}q^{\eps-\delta_2} \\ |y_2| \ll \sqrt{\frac{m_1m_2^2}{c_\theta }} q^{\eps-\delta_2}  \\ |y_3| \ll \sqrt{\frac{m_1m_3^2}{c_\theta }}q^{\eps-\delta_2} \end{cases} } & \text{ if } i=2 \text{ and } m_1m_2m_3/c_\theta >1/4 \\ \R^3_{>0} & \text{ if } i=2 \text{ and } m_1m_2m_3/c_\theta \leq 1/4 .\end{cases}}    
We may swap orders of integration, complete the square, and evaluate the integral over the $v_i$ to get \es{\label{ftL1}L_i= \frac{e(-1/8)}{\sqrt{2}} e\left(-\frac{m_1m_2m_3}{c_\theta }\right) \iiint_{\widehat{R}_i}   \widehat{g_1}(\mathbf{y}) e\left(\frac{y_1^2}{4}-y_2y_3\right)\,d\mathbf{y} + O_{\eps,\delta}( q^{-2014}).}  Now we take $\eps= \delta/4=\delta_2/4$ and use Taylor expansions.  Each of $|y_1|$ and $|y_2y_3|$ are restricted to be small in \eqref{ftL1}, and we need only finitely many terms of the Taylor expansion.  Quantitatively, if $m_1m_2m_3/c_\theta>1/4$ we take $N_1=N_{23} = 8065 \delta^{-1}$ to obtain \est{ e\left(\frac{y_1^2}{4}-y_2y_3\right) = \sum_{n_1=0}^{N_1}\sum_{n_{23}=0}^{N_{23}} \frac{ (\pi i y_1^2/2)^{n_1}(2\pi i y_2y_3)^{n_{23}}}{n_1!n_{23}!} + O_\delta\left(q^{-2016}\right).}  
We now have \est{L_i= \frac{e(-1/8)}{\sqrt{2}} e\left(-\frac{m_1m_2m_3}{c_\theta }\right)\sum_{\substack{ 0 \leq n_1 \leq N_1 \\ 0\leq n_{23} \leq N_{23} }} \frac{ (\pi i /2)^{n_1} (2\pi i )^{n_{23}}}{n_1!n_{23}!} \iiint_{\widehat{R}_i} \widehat{g_1}(\mathbf{y})y_1^{2n_1}(y_2y_3)^{n_{23}}\,d\mathbf{y}   \\ + O_{\delta}( q^{-2014}).}  
 
We now may extend the integrals in the above formulas for $L_1$ and $L_2$ to all of $\R^3$ without making new error terms.  We have that \est{ \iiint_{\R^3 } \widehat{g}_i(\mathbf{y})y_1^{2n_1}(y_2y_3)^{n_{23}}\,d\mathbf{y} = \frac{g_i^{(2n_1,n_2,n_3)}(0)}{(2\pi i )^{2n_1+2n_{23}}},} for $i=1,2$.  In the case $m_1m_2m_3/c_\theta \leq 1/4$ all of the derivatives vanish identically and we therefore ignore this range of $m_1m_2m_3/c_\theta$.    The germs of $g_1(\mathbf{v})$ and $g_2(\mathbf{v})$ at $\mathbf{v}=0$ are identical and so we may drop the distinction between these two functions.  Let e.g. $\delta = 1/10$.  We therefore have that whenever all $m_i>0$ and $m_1m_2m_3/c_\theta > 1/4$ that \es{\label{U+series}\check{U}_\alpha^+ = \frac{e(-1/8)}{\sqrt{2}} e\left(-\frac{m_1m_2m_3}{c_\theta }\right)\sum_{ \substack{0 \leq n_1 \leq N_1 \\ 0 \leq n_{23} \leq N_{23}}} \frac{g^{(2n_1,n_{23},n_{23})}(0) }{4^{n_1}(2 \pi i)^{n_1+n_{23}}n_1!n_{23}!}  + O(q^{-2014}) } with $N_1=N_{23} = 80650.$ If $m_1m_2m_3/c_\theta\leq 1/4$ then $\check{U}^+_\alpha$ is identically 0.

The same exact argument works with minor changes to evaluate $\check{U}^-_\alpha$ in the case that all $m_i<0$.  In particular we define the change of variables $\varphi$ replacing every instance of $m_i$ with $|m_i|$ to find that \est{f^-(\varphi(\mathbf{v})) = -\frac{m_1m_2m_3}{c_\theta} +v_1^2-v_2v_3.} The rest of the proof goes through as above with absolute values added around each $m_i$ in the definition of $g$ and we get exactly the same formula as \eqref{U+series} for $\check{U}^-_\alpha$ but with $e(1/8)$ in place of $e(-1/8)$.  

Now we return to our initial decomposition \eqref{decomposition1} and insert \eqref{U+series} for $\check{U}^\pm_\alpha$.  We evaluate the integrals over $\theta$ using the following 1-variable stationary phase Lemma.
\begin{lemma}\label{statphase2}  Let $s_{u}(\theta)$ be a $C^\infty([0,\pi])$ function of $\theta$ for fixed $u$ and a  holomorphic function of $u$ for $\real(u)> -1/2$.  Let $M\in \R_{>0}$ and suppose that \begin{enumerate}[(i)] \item For fixed $\theta$ the function $s_{u}(\theta)$ is polynomially bounded in vertical strips $\real(u) \geq \delta >-1/2$ for any such fixed $\delta$.  \item The function $s_u(\theta)$ is bounded and non-oscillatory as $\theta \rightarrow 0$ or $\pi$.  Specifically, \est{(\sin \theta)^n \frac{d^n}{d\theta^n} s_{u}(\theta) \ll_{n,u} (\sin^2 \theta)^{\real(u) + 1/2},} where the dependence on $u$ is polynomial in vertical strips, and uniform as $\real(u)\geq \delta >-1/2$ for any such fixed $\delta$.\end{enumerate}  Define
\est{ S_{u}(M)\eqdef \int_0^\pi s_{u}(\theta) w(M^2/\sin^4 \theta) e(-M \cot^2 \theta)\,d\theta,} where $w(x)$ is any smooth bounded non-oscillatory function on $\R_{>0}$ which is moreover identically 0 for $x<1/2$.  Then the function $S_{u}(M)$ is 
\begin{enumerate}[(a)] \item holomorphic and polynomially bounded as a function of $u$ for fixed $M$ \item $C^\infty(\R_{>0})$ as a function of $M$ for each fixed $u$ and non-oscillatory, i.e. satisfies \est{M^n \frac{d^n}{dM^n} S_u(M)\ll_{n,u} 1,} with polynomial dependence on $u$ in vertical strips.  \item  \est{ \begin{cases} \ll M^{1 + \real(u)} & M\ll 1\\  \frac{s_{M,u}(\pi/2)\sqrt{\pi}e(1/8)}{\sqrt{M}} +O_u(M^{-1}) & M\gg 1.\end{cases}} \end{enumerate}
\end{lemma}
\begin{proof}
The bounds in the hypothesis of the Lemma are given in terms of $\sin \theta$ because we are concerned with the behavior as $\theta \rightarrow 0 $ or $\pi$.  Near 0 we approximate $\sin \theta$ by $\theta$ and $\cot^2 \theta $ by $\theta^{-2}$.  Those $\theta$ near $\pi$ are treated symmetrically.

First, consequence (a) is easily established by the absolute convergence of the integral $S_u(M)$.  To establish consequences (b) and (c) we split the integral into three ranges using a smooth partition of unity \es{\label{sdecomp}S_u(M) = \left(\int_{{\rm I}} +  \int_{{\rm II}} + \int_{{\rm III}} \right) s_{u}(\theta) w(M^2/\sin^4 \theta) e(-M \cot^2 \theta)\,d\theta,} where the support of each component of the partition of unity is restricted to \est{{\rm I} \eqdef \{\theta | \sin \theta \leq \min (8 M^{1/2+\eps},1/25)\} } \est{{\rm II} \eqdef \{\theta | \min(4M^{1/2+\eps},1/50) < \sin \theta \leq \min (4 M^{1/2},1/50)\} } \est{{\rm III} \eqdef \{\theta | \sin \theta \geq \min (2 M^{1/2},1/100)\}. } The value of $\eps$ will be chosen in terms of the number of derivates $n$ we take to establish (b).  To establish (b) and (c) we need to control the behavior of $S_u(M)$ both as $M\rightarrow 0$ or $\infty$, and note that the range II is empty when $M\rightarrow \infty$ and that on range III the integral is identically 0 if $M \rightarrow 0$.  

First we consider the range III, and assume $M>1$.  We may differentiate under the integral as many times as we like, and thus this part of the integral is $C^{\infty}$.  As $M\rightarrow \infty$, we differentiate under the integral and run a standard stationary phase argument, say Proposition 8.2 of \cite{BKY}, on each of the derivatives.  Each differentiation produces a factor of $-\cot^2 \theta$ in the integrand which gives additional zeros at $\theta=\pi /2$.  By stationary phase these zeros give us additional decay in each derivative as $M\rightarrow \infty$, demonstrating parts (b) and (c). 

Next we consider the range II, and assume $M<1$, otherwise the set II is empty.  We may differentiate under the integral as often as we like and trivial bounds show that this part of \eqref{sdecomp} is $C^\infty$ and $O(M^{1+\real(u)})$ as $M\rightarrow 0$, matching the claim in (c).  To show (b) let us differentiate in $M$ under the integral $n$ times.  If $2n \leq 2\real(u)+1$ then trivial bounds suffice to show $M^n \frac{d^n}{dM^n}S_u(M) \ll_n 1$.  When $2n>2 \real(u)+1$ we take $\eps=1/4n$ to find that $M^n \frac{d^n}{dM^n}S_u(M) \ll_n M^{1/2+1/2n}\ll_n 1$ when $M<1$.  

Lastly we consider the range I.  In this range we may \textit{not} differentiate under the integral.  Instead we insert a dyadic partition of unity to the integral over the range I which descends into the trouble points $0$ and $\pi$.  Consider the point 0.  We let our partition of unity by formed by smooth functions $W(x)$ supported in $[1/4,1]$ satisfying $x^j W^{(j)}(x) \ll_j1$.  Then the half of the integral in the range I near 0 is \es{\label{theta dyads} \sum_{i=0}^\infty \int_{2^{-i}\min/4}^{2^{-i}\min} W(\theta 2^{i} \min )s_{u}(\theta) w(M^2/\sin^4 \theta) e(-M \cot^2 \theta)\,d\theta,} where $\min = \min (8 M^{1/2+\eps},1/25)$.  Let $\Theta_i = 2^{-i}\min.$  Then we may differentiate each integral in \eqref{theta dyads} in $M$ under the integral sign $n$ times, getting \es{\label{onediffthetadyad} \int_{\Theta_i/4}^{\Theta_i} W(\theta 2^{i} \min )s_{u}(\theta) w(M^2/\sin^4 \theta)(-2\pi i \cot^2 \theta)^n e(-M \cot^2 \theta)\,d\theta,} which converges absolutely.  Now we may apply Lemma \ref{BKYlemma} with $X=\Theta_i^{2 \real(u) +1-2n}, U=\Theta_i, R= M/\Theta_i^3, Y=M/\Theta_i^2, Q=\Theta_i$, to find that each of the integrals \eqref{onediffthetadyad} is \es{\label{thetadyadsumbound} \ll_A \Theta_i^{2(\real(u)+1-n)} (M/\Theta_i^2)^{-A},} for any $A>0$.  Therefore the infinite series \eqref{theta dyads} converges absolutely, the interchange of summation and differentiation is justified, and one easily establishes the lemma for range I.  Putting the ranges I,II and III together we conclude the Lemma.   
\end{proof}

Now we use Lemma \ref{statphase2} to evaluate the combination of \eqref{decomposition1} and \eqref{U+series} for both $\check{U}^\pm_\alpha$.  Take for example the first term of \eqref{U+series}:  \est{g(0) = &\frac{w_1((m_1m_2m_3/c_\theta)^2)}{\sqrt{c_\theta m_1m_2m_3}} \frac{1}{(2\pi i )^3} \iiint_{(1/16)} \prod_{i=1}^3\frac{\Gamma(\kappa/2+u_i)(2\pi)^{-u_i}}{(u_i-\alpha_i)}\zeta_q(1+u_2+u_3) \\ & \times \left(\frac{m_2m_3}{q\sin^2 \theta}\right)^{-u_1}\left(\frac{m_1m_3}{q\sin^2 \theta}\right)^{-u_2}\left(\frac{m_1m_2}{q\sin^2 \theta}\right)^{-u_3}\,du_1\,du_2\,du_3 .} We exchange the three $u_i$ integrals here with the theta integral from \eqref{U+series}.  Abbreviate $m_1m_2m_3/c =M$ and artificially multiply by $e(-M)e(M)$, bringing the latter of these two inside the $\theta$ integral.  The resulting integral to be evaluated is \est{ \int_0^\pi \sin((\kappa-1)\theta) (\sin^2 \theta)^{u_1+u_2+u_3} w_1(M^2/\sin^4 \theta)e(-M \cot^2 \theta)\,d\theta,} to which we apply Lemma \ref{statphase2} with $s_u(\theta) = \sin((\kappa-1)\theta) (\sin^2 \theta)^{u_1+u_2+u_3},$ and $w=w_1$.  By Lemma \ref{statphase2} the resulting function of $M$ is bounded by $\ll \min(M^{1+\real(u)},M^{-1/2}),$ and has controlled derivatives.  Therefore we may take its Mellin transform to obtain a term of the form $T_1$ in the statement of Proposition \ref{Oscillatory}.  

More generally, for the higher terms in the Taylor series \eqref{U+series} one has \est{ g^{(2n_1,n_{23},n_{23})}(0) =\frac{w_{n_1,n_{23}}^*((m_1m_2m_3/c_\theta)^2)}{\sqrt{c_\theta m_1m_2m_3}} \left( \frac{c_\theta}{m_1m_2m_3}\right)^{n_1+n_{23}} \frac{1}{(2\pi i )^3} \iiint_{(1/16)} \zeta_q(1+u_2+u_3) \\ \times \frac{U_{n_1,n_{23}}(u_1,u_2,u_3) }{(u_1-\alpha_1)(u_2-\alpha_2)(u_3-\alpha_3)}\left(\frac{m_2m_3}{q\sin^2 \theta}\right)^{-u_1}\left(\frac{m_1m_3}{q\sin^2 \theta}\right)^{-u_2}\left(\frac{m_1m_2}{q\sin^2 \theta}\right)^{-u_3}\,du,} for some holomorphic functions $U_{n_1,n_{23}}$ which decay rapidly in the vertical strips $\real(u_i)>-\kappa/2$, and $w_{n_1,n_{23}}^*$ are some bounded $C^{\infty}$ functions which are identically zero when their arguments are $<1/2$.  Therefore we may apply Lemma \ref{statphase2} to amplitude functions of the form \est{s_u (\theta) = (\sin^2 \theta)^{n_1+n_{23} + u} \sin((\kappa-1)\theta).}  In this fashion we evaluate each term of the series \eqref{U+series} for $\check{U}^+_\alpha$ separately and add all of the finitely many terms of the form $T_1$ together to conclude that the second term of \eqref{decomposition1} in the case all $m_i>0$ is equal to $T_1$ plus a small error term.  The treatment of the $\check{U}^-_\alpha$ term when all $m_i <0$ is very similar. This concludes our use of the stationary phase method.

We still need to give the Mellin transforms for several remaining cases: the term of \eqref{decomposition1} containing $\check{U}^+_\alpha$ when all $m_i<0$, the corresponding $\check{U}^-_\alpha$ term when all $m_i>0$, and the term $\check{R}_\alpha$. We compute these transforms directly and all of these terms are of the form $T_2$.

We start with the term of \eqref{decomposition1} containing $\check{U}^+_\alpha$ and the assumption that all $m_i<0$. Let $\widetilde{w_1}(v)$ denote the Mellin transform of $w_1$, which is convergent and rapidly decaying in vertical strips whenever $\real(v)<0$.  We have then that for any choice of $\real(v)=\ell<0$ the Mellin transform of $w_1(x)e(2\sqrt{x})$ is given by \est{  \int_0^\infty w_1(x) e(2 \sqrt{x}) x^{s}\,\frac{dx}{x} = \frac{1}{2\pi i }\int_{(\ell)} \frac{ 4^{-(s-v)} \Gamma(2(s-v))}{(-2\pi i )^{2(s-v)} } \widetilde{w_1}(v)\,dv \eqdef \frac{1}{(-2\pi i )^{2s}} u(s).}  The Mellin transform defines a holomorphic function $u(s)$ for any $\real(s)<1/2$ that is rapidly decaying in vertical strips.  A change of variables $s\rightarrow 1/2-s$ and \eqref{Ve} show that when all $m_i<0$ we have \est{\check{U}^+_\alpha  = \frac{(-4\pi i )^{-1}}{c_\theta} \frac{1}{2\pi i } \int_{(3/4)} \frac{u(1/2-s)}{(-2\pi i )^{-2s} } \left(\frac{c_\theta}{|m_1m_2m_3|}\right)^s\frac{1}{(2\pi i )^3} \iiint_{(1/16)} \prod_{i=1}^3 \frac{\Gamma(\kappa/2+u_i)}{(u_i-\alpha_i)} \\ \times\left(\frac{|m_i|q}{2\pi c}\right)^{u_i}  \frac{\Gamma(s-u_i)}{(-2\pi i )^{s-u_i}}\zeta_q(1+u_2+u_3)\,du_1\,du_2\,du_3\,ds.} The parts of the above integrand which depend on $s$ are \est{\frac{u(1/2-s)}{(-2\pi i)^{-2s} } \prod_{i=1}^3 \frac{\Gamma(s-u_i)}{(-2\pi i )^{s-u_i}} = \frac{u(1/2-s) \Gamma(s-u_1)\Gamma(s-u_2)\Gamma(s-u_3)}{(-2\pi i )^s}. } We make a change of variables $s\rightarrow s+u_1+u_2+u_3$ and collect the resulting function $U_3(s,u_1,u_2,u_3)$ which is holomorphic and rapidly decaying in any vertical strips satisfying \est{ \real(s)> & \max_{i\neq j}(-\real(u_i+u_j),-\real(u_1+u_2+u_3)) \\ \real(u_i)>& -\kappa/2.}  We have that \est{\check{U}^+_\alpha = \frac{1}{c} \frac{1}{(2\pi i)^4} \int_{(3/4)}\iiint_{(1/16)} \frac{U_3(s,u_1,u_2,u_3)\zeta_q(1+u_2+u_3)}{(u_1-\alpha_1)(u_2-\alpha_2)(u_3-\alpha_3)} (\sin^2 \theta)^{s+u_1+u_2+u_3-1}  \\ \times \left(\frac{c}{|m_1m_2m_3|}\right)^s \left(\frac{|m_2m_3|}{q}\right)^{-u_1}\left(\frac{|m_1m_3|}{q}\right)^{-u_2}\left(\frac{|m_1m_2|}{q}\right)^{-u_3}\,du\,ds.}  Finally, we re-introduce the integral over $\theta$.  By e.g. \cite{GR7} formulas 3.631.1 and 8.384.1 we have when $\real(z)>0$ that \est{2\chi(-1) \int_0^\pi e(-(\kappa-1)\theta/2\pi) (\sin^2 \theta)^{z-1/2}\,d\theta = 2\chi(-1) \frac{2^{-2z} \Gamma(2z)}{\Gamma(z+\kappa/2)\Gamma(z-\kappa/2+1)}.}  Setting $z=s+u_1+u_2+u_3$ and \est{U_4 (s,u_1,u_2,u_3) \eqdef   2\frac{2^{-2z} \Gamma(2z)}{\Gamma(z+\kappa/2)\Gamma(z-\kappa/2+1)}U_3 (s,u_1,u_2,u_3),} which is holomorphic and rapidly decaying in the same vertical strips as $U_3$.  The result is that when all $m_i<0$ the term \est{ 2\chi(-1) \int_0^\pi e(-(\kappa-1) \theta/2\pi) (\sin \theta)  \check{U}^+_\alpha(\theta)\,d\theta } is of the form $T_2$ in the statement with $U_4$ in place of $U_2$.  The integral over $\theta$ involving $\check{U}^-_\alpha$ when all $m_i>0$ is treated similarly.

Next we compute the Mellin transform of the term $\check{R}_\alpha$ from the decomposition \eqref{decomposition1}.  Recall that the function $R(x)$ defined at the beginning of the proof of Proposition \ref{Oscillatory} is identically equal to $J(x)$ when $x<\sqrt{2}$, satisfies the bound $x^n R^{(n)}(x) \ll_n 1$ for all $n\in \N$, and is $\ll_\kappa x^{-2}$ for large $x$.  Let $\widetilde{R}$ denote the Mellin transform of $R$.  Then with a change of variables we have the inversion formula \es{\label{Rcheck} R(2\sqrt{cx_1x_2x_3}) = \frac{\chi(-1)}{2} \frac{1}{2\pi i} \int_{(1/2)} 4^s \widetilde{R}(2-2s) (cx_1x_2x_3)^{s-1}\,ds. } the integrand of \eqref{Rcheck} is rapidly decaying on any vertical strip $0< \real(s) < \kappa/2$, with a simple pole at $\kappa/2$.  Then inserting formula \eqref{Rcheck} in the definition of $\check{R}_\alpha$ and evaluating the $x$ integrals with \eqref{Ve} we have that \est{\check{R}_\alpha = \frac{\chi(-1)}{2c} \frac{1}{2\pi i } \int_{(3/4)} 4^s \widetilde{R}(2-2s)\left(\frac{c}{|m_1m_2m_3|}\right)^s \frac{1}{(2\pi i)^3} \iiint_{(1/16)} \prod_{i=1}^3 \frac{\Gamma(\kappa/2 + u_i)}{(u_i-\alpha_i)} \left(\frac{|m_i|q}{2\pi c}\right)^{u_i}\\ \times \frac{\Gamma(s-u_i)}{( 2 \pi i  \sgn m_i )^{s-u_i}}\zeta_q(1+u_2+u_3)\,du_1\,du_2\,du_3\,ds.} Change $s\rightarrow s+u_1+u_2+u_3$ and set \est{U_5(s,u_1,u_2,u_3)\eqdef \frac{1}{2} 4^{s+u_1+u_2+u_3} \widetilde{R}(2-2(s+u_1+u_2+u_3))\prod_{i=1}^3 \Gamma(\kappa/2 + u_i) \left(2\pi \right)^{-u_i}\\ \times \frac{\Gamma(s+u_1+u_2+u_3-u_i)}{( 2 \pi i \sgn m_i )^{s+u_1+u_2+u_3-u_i}}, } which is which is holomorphic and rapidly decaying in the vertical strips allowed by \eqref{regionvalid}.  Then we have that $\check{R}_\alpha$ is of the form $T_2$ with $U=U_5$. Setting $U_2=U_4+U_5$ if all $m_i$ are of the same sign, this concludes the proof of Proposition \ref{Oscillatory} if all $m_i>0$ or all $m_i<0$.  

Finally we compute the Mellin Transform of $\check{W}_\alpha$ when the $m_i$ are of mixed signs.  Suppose without loss of generality that $m_1<0$ and $m_2,m_3>0$; the other 5 mixed-sign cases are treated similarly.  Recall the formulas \eqref{MellinJ} and \eqref{Ve} from section \ref{Diagonal}. 
Putting these together we find \est{\check{W}_\alpha = \frac{\chi(-1)}{c} \int_{(3/4)} (2\pi)^{2s} \frac{\Gamma(\kappa/2-s)}{\Gamma(\kappa/2 +s) } c^s \frac{1}{(2\pi i)^3 } \iiint_{(1/16)} \prod_{i=1}^3 \frac{\Gamma(\kappa/2 + u_i)}{(u_i-\alpha_i)} \left(\frac{q}{2\pi c}\right)^{u_i} \frac{\Gamma(s-u_i)}{(2 \pi i m_i)^{s-u_i}} \\ \times \zeta_q(1+u_2+u_3)\,du_1\,du_2\,du_3\,ds.}  Observe that \est{\prod_{i=1}^3\frac{\Gamma(s-u_i)}{(2 \pi i m_i)^{s-u_i}} = \frac{|m_1|^{u_1}|m_2|^{u_2}|m_3|^{u_3}}{|m_1m_2m_3|^s}\frac{\Gamma(s-u_1)}{(2 \pi)^s (-2 \pi i )^{-u_1}}\frac{\Gamma(s-u_2)}{(2\pi )^s(2 \pi i )^{-u_2}}\frac{\Gamma(s-u_3)}{(2 \pi i )^{s-u_3}}} is a rapidly decaying function of $s$ in vertical strips uniformly in $u_i$.  We set \est{U_6(s,u_1,u_2,u_3) \eqdef \frac{\Gamma(\kappa/2-s)}{\Gamma(\kappa/2 +s) }  \prod_{i=1}^3 \Gamma(\kappa/2 + u_i) (2 \pi )^{-u_i} \frac{\Gamma(s-u_1)}{ (-2 \pi i )^{-u_1}}\frac{\Gamma(s-u_2)}{(2 \pi i )^{-u_2}}\frac{\Gamma(s-u_3)}{(2 \pi i )^{s-u_3}},} and after a 
change of variables $s\rightarrow s+u_1+u_2+u_3$ set \est{U_7(s,u_1,u_2,u_3)\eqdef U_6(s+u_1+u_2+u_3,u_1,u_2,u_3).}  The function $U_7$ is holomorphic and rapidly decaying in any vertical strips \eqref{regionvalid}.  We get when $m_i$ are of mixed signs that $\check{W}_\alpha$ is of the form $T_2$ with $U=U_7$ if $m_1<0$ and $m_2,m_3>0$.  The other mixed-sign cases have similar formulas. \end{proof}

\section{A Large Sieve Inequality for $L$-functions}\label{Lsieve}

\begin{lemma}[Large Sieve for $L$-functions]\label{sieve}
We have for $q$ odd square-free and $t_1, t_2$ with $|t_1|,|t_2| \ll q^{100}$ and $|\imag(t_1,t_2) |\leq 1/\log q$ that \es{\label{lem}\frac{1}{\phi(q)} \sum_{\psi \pmod q} \left| L(1/2+it_1,\psi)L(1/2+it_2,\psi)\right|^2  \ll  (\log q)^2 \left(\log  (\omega(q)+1)\right)^2 \\  \times \begin{cases} (\log q)^2 & \text{ if } |t_1-t_2|\leq 1/\log q \\ |\zeta(1+i|t_1-t_2|)|^2 & \text{ if }  1/\log q < |t_1-t_2| . \end{cases}  }  
\end{lemma}
\begin{proof}
Note that we have stated the lemma as a sum over all characters modulo $q$, both primitive and imprimitive. If $\psi$ is not primitive, let $q^*$ denote the conductor of $q$ and let $\psi^*$ denote the primitive character mod $q^*$ which induces $\psi$.  Then for an non-primitive character $\psi$ we have \est{L(s,\psi) = \prod_{p \mid q/q^*} \left(1- \psi^*(p) p^{-s}\right) L(s,\psi^*).}  Let $\psum$ and $\msum$ denote sums over primitive even (resp. odd) characters.  One has by orthogonality of characters and positivity that \es{\label{primitivesieve} \psum_{\psi \pmod q} \left| \sum_{n \asymp N} a_n \psi(n) \right|^2 \ll \left(q+N\right)\sum_{n \asymp N} |a_n|^2,} where $n\asymp N$ denotes a sum over $n$ whose length is $\ll N$ and $\gg N$, and similarly for $\msum$.  We can split the moment in the statement of the Lemma over primitive characters of a given parity to find that the left hand side of \eqref{lem} is bounded by  \es{\label{ref}\leq \frac{1}{\phi(q)} \sum_{q^* \mid q}\prod_{p \mid q/q^*} \left(1+ \frac{1/e}{\sqrt{p}}\right)^4 \left(\psum_{\psi^* \pmod {q^*}} + \msum_{\psi^* \pmod {q^*}}\right) \left| L(1/2+it_1,\psi^*)L(1/2+it_2,\psi^*)\right|^2.}  Next we write an approximate functional equation for the product of two $L$-functions.  Let $\sigma_{t_1,t_2}(n) = \sum_{n=n_1n_2} n_1^{-it_1}n_2^{-it_2}$.  Following say, Theorem 5.3 of \cite{IK} we have that \est{L(1/2+it_1,\psi^*)L(1/2+it_2,\psi^*) = \sum_{n\geq 1} \frac{\sigma_{t_1,t_2}(n)\psi^*(n)}{n^{1/2} } F_{t_1,t_2}(n/q^*) \\ + \eps(\psi,t_1,t_2) \sum_{n\geq 1}  \frac{\sigma_{-t_1,-t_2}(n)\overline{\psi^*}(n)}{n^{1/2} } F_{-t_1,-t_2}(n/q^*),} where $|\eps(\psi,t_1,t_2)|=1$ and $F_{t_1,t_2}(x)$ is a $C^\infty$ and rapidly decaying function of $x$ for fixed $t_1,t_2$ and bounded in vertical strips as a function of $t_1$ and $t_2$. By Cauchy-Schwarz we have \est{ \left|L(1/2+it_1,\psi^*)L(1/2+it_2,\psi^*)\right|^2 \\ \leq 2 \left|\sum_{n\geq 1} \frac{\sigma_{t_1,t_2}(n)\psi^*(n)}{n^{1/2} } F_{t_1,t_2}(n/q^*)\right|^2 + 2\left|\sum_{n\geq 1}  \frac{\sigma_{-t_1,-t_2}(n)\overline{\psi^*}(n)}{n^{1/2} } F_{-t_1,-t_2}(n/q^*)\right|^2.}  
Inserting this to \eqref{ref} we find the left hand side of \eqref{lem} is bound by a sum of four terms, one of which is \es{\label{ref2}  \frac{2}{\phi(q)} \sum_{q^* \mid q}\prod_{p \mid q/q^*} \left(1+ \frac{1/e}{\sqrt{p}}\right)^4 \psum_{\psi^* \pmod q^*}\left|\sum_{n\geq 1} \frac{\sigma_{t_1,t_2}(n)\psi^*(n)}{n^{1/2} } F_{t_1,t_2}(n/q^*)\right|^2,} and the other three are of a similar form.  By the large sieve \eqref{primitivesieve}, the upper bound \eqref{ref2} is \est{\ll  \frac{1}{\phi(q)} \sum_{q^* \mid q}q^*\prod_{p \mid q/q^*} \left(1+ \frac{1/e}{\sqrt{p}}\right)^4 \sum_{n\geq 1} \frac{|\sigma_{t_1,t_2}(n)|^2}{n}\left| F_{t_1,t_2}(n/q^*) \right|^2.}  We evaluate this last sum by the Dirichlet series and Mellin inversion technique.  We have \est{\sum_{n\geq 1} \frac{|\sigma_{t_1,t_2}(n)|^2}{n^s} = \zeta(s+it_1-i\overline{t_1})\zeta(s-i\overline{t_1} + it_2)\zeta(s+it_1-i\overline{t_2})\zeta(s+it_2-i\overline{t_2}),}  
so that \est{\sum_{n\geq 1} \frac{|\sigma_{t_1,t_2}(n)|^2}{n}\left| F_{t_1,t_2}(n/q^*) \right|^2 \ll q^* (\log q^*)^2 \begin{cases} (\log q^*)^2 & \text{ if } |t_1-t_2|\leq 1/\log q^* \\ |\zeta(1+i|t_1-t_2|)|^2 & \text{ if }  1/\log q^* < |t_1-t_2| . \end{cases}} Putting all of these pieces together, the left hand side of \eqref{lem} is \est{ \ll  \left(\sum_{q^* \mid q}{q^*}^{-1}\prod_{p \mid q^*} \left(1+ \frac{1/e}{\sqrt{p}}\right)^4 \right) \frac{q}{\phi(q)}  (\log q)^2 \begin{cases} (\log q)^2 & \text{ if } |t_1-t_2|\leq 1/\log q \\ |\zeta(1+i|t_1-t_2|)|^2 & \text{ if }  1/\log q < |t_1-t_2| . \end{cases}}  The first factor in parentheses is seen by Merten's theorem to be $\ll \log (\omega(q)+1)$, hence we conclude the Lemma.
\end{proof}

\section{The Error Terms: Dual Moment}\label{Dual}

In this section we assemble the remaining off-diagonal sum \est{\sum_{c \equiv 0 \pmod q} \sum_{m_1m_2m_3 \neq 0} c^{-2}G(m_1,m_2,m_3,c) \check{W}_\alpha(m_1,m_2,m_3,c)} into an average of Dirichlet $L$-functions of conductors dividing $q$ on their critical lines, as per the ``Motohashi-type formula'' from the introduction.  We assume $\kappa\geq 4$ throughout section \ref{Dual} and give details on the case $\kappa =2$ in section \ref{two}.

The calculations below are quite intricate and involve large expressions because they are valid for all odd square-free $q$.  The calculations become much cleaner fashion if one takes $q$ to be prime and we suggest this assumption on a first reading.  
\begin{proof}[Proof of Main Theorems]
We computed the analytic component $\check{W}_\alpha$ in Proposition \ref{Oscillatory}.  The evaluation of the arithmetic component $G$ by Conrey and Iwaniec is sufficient for our purposes.  We recall their Lemma 10.2 as Lemma \ref{Gcalc} below. Let $R_q(n) = S(0, n,q)$ be the Ramanujan sum, and define $H(w;q)$ to be the character sum \est{ H(w;q) \eqdef \sum_{u,v \pmod q} \chi(uv(u+1)(v+1)) e_q((uv-1)w).}
\begin{lemma}[Conrey-Iwaniec]\label{Gcalc}
Let $c=qr$ with $q$ square-free.  Suppose $m_1,m_2,m_3$ are integers with \es{\label{coprimeconditions}(m_1,r)=1\,\,\,(m_2m_3,q,r)=1.}  Then the character sum $G(m_1,m_2,m_3,c)$ satisfies \est{G= e_c(m_1m_2m_3)\chi_{k\ell}(-1) \frac{r^2qh}{\phi(k)} R_k(m_1)R_k(m_2)R_k(m_3) H(\overline{rhk}m_1m_2m_3;\ell)} where $h=(r,q), k=(m_1m_2m_3,q)$ and $\ell = q/hk$.  If the co-primality conditions \eqref{coprimeconditions} are not satisfied then $G$ vanishes.
\end{lemma}

Following Conrey and Iwaniec section 11 we expand the character sum $H$ in multiplicative characters.  The formulas (11.7), (11.9) and (11.10) of Conrey and Iwaniec together give \es{\label{Hdecomp}H(w;q) = \sum_{q_1q_2=q} \frac{\mu(q_1)\chi_{q_1}(-1)}{\phi(q_2)} \sum_{\psi \pmod {q_2}} \tau(\overline{\psi})g(\chi,\psi)\psi(\overline{q_1} w),} where $g(\chi,\psi)$ is the hybrid character sum in the statement of Theorem \ref{Thm2}. 
\begin{lemma}\label{Deligne}
For all Dirichlet characters $\psi$ of odd square-free modulus $q$ we have that \est{g(\chi,\psi) \leq q C^{\nu(q)},} where $\nu(q)$ is the number of prime factors of $q$.  
\end{lemma}
\begin{proof} Conrey and Iwaniec prove that if $q$ is a prime that $g(\chi,\psi) \leq C q$ for an absolute constant $C$.  The character sum is multiplicative in the modulus, hence $g(\chi,\psi)\leq C^{\nu(q)}q.$  One could give an explicit value to $C$ using the exponential sum techniques of Katz.  \end{proof}

The expressions for $\check{W}_\alpha$ and $G$ have symmetry under $m_i\rightarrow -m_i$, so we now assume that all $m_i>0$.  The other 7 octants of $m_i$'s are treated separately and in the same way and then we add all contributions together at the end.  Of the three terms in the evaluation of $\check{W}_\alpha$ given by Proposition \ref{Oscillatory}, the contribution of the error term is negligible and we ignore it.  Terms $T_1$ and $T_2$ of Proposition \ref{Oscillatory} are similar, but $T_2$ is more difficult and contains the issue of small weights, so we only give full details for $T_2= T_2(m_1,m_2,m_3,c)$.  That is to say, we restrict our attention to the sum \es{\label{total}\sum_{r\geq 1} \sum_{\substack{m_i \geq 1\\ i=1,2,3}} \frac{G(m_1,m_2,m_3,rq)T_2(m_1,m_2,m_3,r,q)}{r^2q^2}} for the remainder of this section.  

The first step is to separate the sums over $r$ and $m_1,m_2,m_3$ according to the parameters $h$ and $k$ in Lemma \ref{Gcalc}.  Starting with $r$, we set $r=hr'$ with $(r',q/h)=1$.  We also split the sum over the $m_i$ according to $k=(m_1m_2m_3,q)$, and we understand the condition $(m_1m_2m_3,q)=k$ as $(m_1m_2m_3,k)=k$ and $(m_i,q/k)=$ for each $i=1,2,3$.  We have that \est{\sum_{r\geq 1} \sum_{\substack{m_i\geq 1\\ (m_1,r)=1 \\ (m_2m_3,q,r)=1}}  = \sum_{hk \mid q} \sum_{\substack{r'\geq 1 \\ (r',q/h)=1}} \sum_{\substack{(m_1m_2m_3,k)=k \\ (m_1,r')=1\\ (m_1m_2m_3,qh/k)=1}}.}  For $k$ square-free there is a bijection between the sets \est{\{m_1,m_2,m_3 \text{ such that } (m_1m_2m_3,k)=k\} \longleftrightarrow \{(k_1,n_1),(k_2,n_2),(k_3,n_3)  \text{ such that } \\ k_i \mid k,\, k \mid k_1k_2k_3 \text{ and } (n_i,k/k_i)=1\}} given by $k_i=(m_i,k)$ and $m_i=k_in_i$.  We split the sums further according to this bijection.  \es{\label{hkdecomp} \sum_{r\geq 1} \sum_{\substack{m_i\geq 1\\ (m_1,r)=1 \\ (m_2m_3,q,r)=1}}  = \sum_{hk \mid q} \sum_{\substack{r'\geq 1 \\ (r',q/h)=1}} \sum_{\substack{k_i \mid k \\ k \mid k_1k_2k_3 \\ (k_1,r')=1 \\ (k_1k_2k_3, qh/k)=1}} \sum_{\substack{(n_1,k/k_1)=1 \\ (n_1,qr'h/k)=1}}  \sum_{\substack{(n_2,k/k_2)=1 \\ (n_2,qh/k)=1}}\sum_{\substack{(n_3,k/k_3)=1 \\ (n_3,qh/k)=1}} ,} where $r=hr'$, $m_i=k_in_i$ for $i=1,2,3.$  
We may now combine the formulas for $T_2$ and $G$ with this decomposition to get a large but tractable formula.  

In $T_2$ we shift the lines of integration to $\real(s)=-1/16$ and $\real(u_i) = 5/8$ so that all of the Dirichlet series and Euler product calculations below occur in the region of absolute convergence.

We decompose the exponential $e_c(m_1m_2m_3)$ by treating it as an archimedian character.  Note here if we were dealing instead with $T_1$ of Proposition \ref{Oscillatory} instead of $T_2$ this phase would cancel cleanly with the phase from $T_1$, making the $T_1$ case easier.  As it is, we use the formula \es{\label{archarformula}e\left(\frac{m_1m_2m_3}{c}\right) = 1+ \frac{1}{2\pi i}\int_{(-1/16)} \frac{\Gamma(v)}{(-2\pi i )^v} \left(\frac{m_1m_2m_3}{c}\right)^{-v}\,dv,} introducing an additional variable. 

The two terms in \eqref{archarformula} may be treated separately.  In fact the contribution from the term 1 of \eqref{archarformula} is of the same form as those terms where we take $T_1$ instead of $T_2$.  As this case is very similar and easier we ignore it.   
Let $G'(m_1,m_2,m_3,c)$ be defined as the expression of Lemma \ref{Gcalc} with the second term of \eqref{archarformula} in lieu of $e_c(m_1m_2m_3)$.  Warning: our $G'$ differs from that of Conrey and Iwaniec, but is similar in spirit.  The modified sum \es{\label{79mod}  \sum_{r\geq 1} \sum_{\substack{m_i \geq 1\\ i=1,2,3}} \frac{G'(m_1,m_2,m_3,rq)T_2(m_1,m_2,m_3,r,q)}{r^2q^2}} is then equal to \es{\label{big1prime} \frac{1}{(2\pi i )^5} \int_{(-1/16)}\int_{(-1/16)} \iiint_{(5/8)}  \frac{\Gamma(v)}{(-2\pi i )^v} \frac{U_2(s,u_1,u_2,u_3)\zeta_q(1+u_2+u_3)}{(u_1-\alpha_1)(u_2-\alpha_2)(u_3-\alpha_3)} q^{s+v+u_1+u_2+u_3-2} \\ \times \sum_{hk \mid q} \chi_{h}(-1)\frac{h^{s+v}}{\phi(k)}\sum_{\ell_1\ell_2= \ell} \mu(\ell_1)\chi_{\ell_1}(-1) D(s,v,u_i,h,k,q)\,du_1\,du_2\,du_3\,ds\,dv} where $D=D(s+v,u_i,h,k,q)$ is \es{\label{big2} D \eqdef  \frac{1}{\phi(\ell_2)} \sum_{\psi \pmod {\ell_2}} \tau(\overline{\psi})g(\chi,\psi)\psi(\overline{\ell_1h^2k})  \sum_{\substack{k_i \mid k \\ k \mid k_1k_2k_3 \\ (k_1k_2k_3, qh/k)=1}} \frac{\psi(k_1k_2k_3)}{k_1^{s+v+u_2+u_3} k_2^{s+v+u_1+u_3}k_3^{s+v+u_1+u_2}} \\ \times \sum_{\substack{r'\geq 1 \\ (r',qk_1/h)=1 }} \frac{\psi(\overline{r'})}{r'^{1-(s+v)}} \sum_{\substack{(n_1,k/k_1)=1 \\ (n_1,qr'h/k)=1}} \frac{R_k(k_1n_1)\psi(n_1)}{n_1^{s+v+u_2+u_3}} \sum_{\substack{(n_2,k/k_2)=1 \\ (n_2,qh/k)=1}} \frac{R_k(k_2n_2)\psi(n_2)}{n_2^{s+v+u_1+u_3}} \sum_{\substack{(n_3,k/k_3)=1 \\ (n_3,qh/k)=1}}  \frac{R_k(k_3n_3)\psi(n_3)}{n_3^{s+v+u_1+u_2}} .}   
\begin{lemma}\label{DirichletSeries}
For $\real(w)>1$ and square-free $k$ we have \est{\sum_{\substack{(n_i,k/k_i)=1 \\ (n_1,qh/k)=1}} \frac{R_k(k_in_i)\psi(n_i)}{n_i^{w}} = \mu(k)\mu(k_i)\phi(k_i) L_{\frac{q}{k_i}h }(w,\psi), }where the subscript denotes primes dividing the subscript have been omitted.
\end{lemma}
\begin{proof}
The Ramanujan sum $R_k(n)$ is not multiplicative in $n$ but if $k$ is square-free the function $\mu(k)R_k(n)$ is multiplicative in $n$.  Define \est{\rho_{k_i}(n) \eqdef  \prod_{\substack{p \mid n \\ p \nmid k_i}} \mu(k) R_k(p).} Then we have $R_k(k_i n_i) = \mu(k)\mu(k_i)\phi(k_i) \rho_{k_i}(n_i),$ and expanding into an Euler product establishes the Lemma.   
\end{proof}

Next we evaluate the sum over $r'$ in \eqref{big2} by an Euler product calculation to find \est{\sum_{(r',\frac{q}{h} k_1)=1} \frac{\overline{\psi}(r')}{r'^{1-(s+v)}} L_{\frac{q}{k_1}h r' }(s+v+u_2+u_3,\psi) = \frac{L_{\frac{q}{h}k_1}(1-(s+v),\overline{\psi})}{\zeta_q(1+u_2+u_3)} L_{\frac{q}{k_1}h }(s+v+u_2+u_3,\psi) .}  Note that the zeta function in the denominator here cancels exactly with the zeta function in the numerator of \eqref{big1prime}.  This cancellation represents the resolution of the artificial use of Hecke multiplicativity from section \ref{Initial}.   
Now take the sum over the $k_i$ from \eqref{big2} inside and denote the resulting sum by $K =K(s+v,u_i,q,h,k,\psi)$. \est{K \eqdef \sum_{\substack{k_i \mid k \\ k \mid k_1k_2k_3 \\ (k_1k_2k_3, qh/k)=1}} \frac{\psi(k_1k_2k_3)\mu(k_1)\mu(k_2)\mu(k_3)\phi(k_1)\phi(k_2)\phi(k_3)}{k_1^{s+v+u_2+u_3} k_2^{s+v+u_1+u_3}k_3^{s+v+u_1+u_2}}  \prod_{p \mid k_1} \left(1-\psi(p) p^{-(s+v+u_2+u_3)}\right)^{-1} \\  \times \prod_{p \mid k_2} \left(1-\psi(p) p^{-(s+v+u_1+u_3)}\right)^{-1}\prod_{p \mid k_3} \left(1-\psi(p) p^{-(s+v+u_1+u_2)}\right)^{-1}.} 

\begin{lemma}
For $k$ square-free the function $K$ is holomorphic when $\real(s+v+u_i+u_j)>0$ for all $i\neq j$.  Moreover it satisfies the bound \est{K \ll \left(k/(k,h)\right)^{3-3\real(s+v)-2\real(u_1+u_2+u_3)}.} 
\end{lemma}
\begin{proof} We have \est{|K| \leq & \sum_{\substack{k_i \mid k \\ k \mid k_1k_2k_3 \\ (k_1k_2k_3, qh/k)=1}}  \prod_{p \mid k_1} \left| \frac{p-1}{p^{s+v+u_2+u_3}-\psi(p)}\right|\prod_{p \mid k_2} \left| \frac{p-1}{p^{s+v+u_1+u_3}-\psi(p)}\right|\prod_{p \mid k_3} \left| \frac{p-1}{p^{s+v+u_1+u_2}-\psi(p)}\right|} all of whose terms are positive.  Therefore we may drop the condition $k\mid k_1k_2k_3$ on the sum to get that $|K|$ is  \est{ \leq \prod_{\substack{p \mid k \\ p \nmid \frac{qh}{k}} } \left(1+\left| \frac{p-1}{p^{s+v+u_2+u_3}-\psi(p)}\right|\right) \left(1+\left| \frac{p-1}{p^{s+v+u_1+u_3}-\psi(p)}\right|\right) \left(1+\left| \frac{p-1}{p^{s+v+u_1+u_2}-\psi(p)}\right|\right)} hence the claim. \end{proof}

The result of all of these manipulations is that \es{\label{Dformula2} D =  \zeta_q(1+u_2+u_3)^{-1}\frac{\mu(k)^3}{\phi(\ell_2)} \sum_{\psi \pmod {\ell_2}} \tau(\overline{\psi})g(\chi,\psi)\psi(\overline{\ell_1h^2k}) K \\ \times L_{q/h}(1-(s+v),\overline{\psi}) L_q(s+v+u_2+u_3,\psi)L_q(s+v+u_1+u_3,\psi)L_q(s+v+u_1+u_2,\psi).}

For a primitive Dirichlet character $\psi$ of conductor $\ell$ we write the asymmetric functional equation as \est{ L(w,\psi) = \eps(\psi) X(w) L(1-w,\overline{\psi})} where \est{X_\ell(\tfrac12 + u) \eqdef \left(\frac{\ell}{\pi} \right)^{-u } \frac{\Gamma\left(\frac{\half - u + \mathfrak{a}}{2} \right)}{\Gamma\left(\frac{\half + u + \mathfrak{a}}{2} \right)}\,\,\,\text{ and }\,\,\,\mathfrak{a} \eqdef \begin{cases} 0 & \psi \text{ even} \\ 1 & \psi \text{ odd.}\end{cases} }  

In the case that $\psi$ is imprimitive we let $\psi^*$ be the primitive character of conductor $\ell_2^*$ which induces $\psi$ of modulus $\ell_2$.  
We collect the miscellaneous Euler factors as $P=P(s+v,u_i,q,h,\ell_2,\psi)$, i.e. \est{P \eqdef & \prod_{\substack{p \mid \ell_2 \\ p \nmid \ell_2^*}} \left(1-\overline{\psi^*}(p)p^{-(1-(s+v))}\right) \prod_{\substack{p \mid q/h \\ p \nmid \ell_2}} \left( 1-\overline{\psi}(p)p^{-(1-(s+v))}\right) \\ & \times \prod_{ \substack{p \mid q \\ p \nmid \ell_2}} \left(1-\psi(p) p^{-(s+v+u_2+u_3)}\right) \left(1-\psi(p) p^{-(s+v+u_1+u_3)}\right) \left(1-\psi(p) p^{-(s+v+u_1+u_2)}\right).}  The function $P$ is entire in all of its complex variables, and if $1-(s+v)\geq \eps >0$ and $s+v+u_i+u_j\geq \eps >0$ for some small fixed $\eps$ for all $i\neq j$ then $P\ll 2^{\nu(q/\ell_2^*)}.$  Then we have that \es{\label{4L1} L_{q/h}(1-s,\overline{\psi}) L_q(s+u_2+u_3,\psi) L_q(s+u_1+u_3,\psi) L_q(s+u_1+u_2,\psi) \\ =  \overline{\eps(\psi^*)}X_{\ell_2^*}(s)^{-1} P \,L(s,\psi^*)L(s+u_2+u_3,\psi)L(s+u_1+u_3,\psi)L(s+u_1+u_2,\psi).}  
We now shift the contours in \eqref{big1prime} to $\real(v)=-3/4$, $\real(s) = 5/4$ and $\real(u_i) = 1/\log q,$ encountering simple polar divisors only when $\psi$ in \eqref{Dformula2} is the trivial character at $s+v=0$ and $s+v+u_i+u_j=1$ for each $i \neq j$.  Because $\tau(\overline{\psi}) g(\chi,\psi)$ is of absolute value 1 when $\psi$ is the trivial character, the terms produced by these divisors are negligibly small.  We replace $L(s+v,\psi^*)$ in \eqref{4L1} by $L(s+v,\psi)$ and absorb the resulting finitely many Euler factors into $P$, leaving the estimate $ P\ll 2^{\nu(q/\ell_2^*)}$ unaffected. 

We make a change of variables $s+v\rightarrow s$, and define \es{\label{U2starF} U_2^*(s,u_1,u_2,u_3) \eqdef \frac{1}{2\pi i } \int_{(-3/4)} \frac{\Gamma(v)}{(-2\pi i )^v} U_2(s-v,u_1,u_2,u_3)\,dv.} The function $U^*_2$ is holomorphic in the region \est{\max(\real(u_i))-1 < \real(s+u_1+u_2+u_3) < \kappa/2 \\ -\kappa/2 < \real(u_i),} symmetric in the $u_i$ variables, and by shifting the contour satisfies the bounds \est{U^*_2 \ll_\eps (1+|s|)^{\ell} \exp(-(\pi/2-\eps)|\imag(u_1+u_2+u_3)|)} where \est{\ell = \max\left(\real(s+u_1+u_2+u_3) -(\kappa+1)/2,-3/2\right).} As $\real(s)=1/2$ and $\kappa\geq 4$ we may take $\ell = -3/2$. Pulling together \eqref{big1prime}, \eqref{Dformula2}, \eqref{4L1}, and \eqref{U2starF} we get that \eqref{79mod} is equal to \es{\label{big4} \frac{1}{q^2} \sum_{hk \mid q} \chi_{h}(-1)\frac{\mu(k)^3}{\phi(k)}  \sum_{\ell_1\ell_2= \ell} \mu(\ell_1)\chi_{\ell_1}(-1) \frac{1}{(2\pi i )^4} \int_{(1/2)}\iiint_{(1/\log q)}   \frac{U_2^*(s,u_1,u_2,u_3)q^{s+u_1+u_2+u_3}h^{s}}{(u_1-\alpha_1)(u_2-\alpha_2)(u_3-\alpha_3)} \\ \times\frac{1}{\phi(\ell_2)} \sum_{\psi \pmod {\ell_2}} \tau(\overline{\psi})g(\chi,\psi)\psi(\overline{\ell_1h^2k}) \overline{\eps(\psi^*)} K P  X_{\ell_2^*}(s)^{-1}  L(s,\psi)L(s+u_2+u_3,\psi) \\ \times L(s+u_1+u_3,\psi)L(s+u_1+u_2,\psi) \,du_1\,du_2\,du_3\,ds  + O(q^{-1/2}).}  There are finitely many similar cases: $m_i$ with different signs, the terms $T_1$ from Proposition \ref{Oscillatory} and the term 1 from \eqref{archarformula}, each of which are evaluated in the same way and can be expressed as a formula \eqref{big4} but with a different choice of holomorphic function $U$ in lieu of $U_2^*$.  We add together these cases and specialize \eqref{big4} to the case $q$ prime to obtain Theorem \ref{Thm2}.

Now we finish the proof of Theorem \ref{MP}.  Taking absolute values inside, \eqref{big4} is bounded by \es{\label{bound} \ll \frac{1}{q^{3/2}} \sum_{hk\mid q} (hk)^{1/2}2^{\nu(hk)} \sum_{\ell_1\ell_2=\ell} \ell_2^{3/2} C^{\nu(\ell_2)}\int_{(1/2)} \frac{1}{(1+|s|)^{3/2}} \\ \times \frac{1}{\phi(\ell_2)}\sum_{\psi \pmod {\ell_2}}   \left| L(s,\psi)L(s+\alpha_2+\alpha_3,\psi)L(s+\alpha_1+\alpha_3,\psi)L(s+\alpha_1+\alpha_2,\psi) \right| \,ds.}

We can apply Cauchy-Schwarz and the large sieve estimate of Lemma \ref{sieve} in three different ways to \eqref{bound}.  The bounds produced by the large sieve in each of these cases depend only on the sizes of \est{\begin{cases} |\alpha_2-\alpha_3|, |\alpha_2+\alpha_3| & \text{ if Cauchy-Schwarz applied as } (12)(34) \\ |\alpha_1-\alpha_2|, |\alpha_1+\alpha_2| & \text{ if Cauchy-Schwarz applied as } (14)(23) \\  |\alpha_3-\alpha_1|, |\alpha_3+\alpha_1| & \text{ if Cauchy-Schwarz applied as } (13)(24), \end{cases} } and do not depend on $s$. Note the symmetry $\alpha_i\rightarrow -\alpha_i$ here, a shadow of the functional equation symmetry in the original problem.  In fact the bounds we obtain depend only on the $|\alpha_i|$ and not on their signs.

Suppose first that all of the $|\alpha_i|$ are within a neighborhood of radius $1/3 \log \ell_2$ of each other.  Then the $\phi(\ell_2)^{-1} \sum_{\psi}$ in \eqref{bound} is bounded by \est{ \ll (\log \ell_2)^2 \left(\log(\nu(\ell_2)+1)\right)^2 \times \begin{cases} (\log \ell_2)^2 & \text{ if all } | \alpha_i| \leq 1/\log \ell_2 \\ (\log \ell_2) |\zeta(1+2i |\alpha_i|)| & \text{ if some } |\alpha_i|>1/\log \ell_2 .\end{cases}}  If not there is a pair of $\alpha_i,\alpha_j$, $i \neq j$ with \es{\label{lastCS}||\alpha_i|-|\alpha_j||>1/\log \ell_2} and we apply Cauchy-Schwarz in the manner which detects the $\alpha_i,\alpha_j$ for which \eqref{lastCS} holds to find that the $\phi(\ell_2)^{-1} \sum_{\psi}$ of \eqref{bound} is \est{\ll (\log \ell_2)^2 \left( \log(\omega(\ell_2)+1)\right)^2 \left|\zeta(1+i|\alpha_i|+i |\alpha_j|)\zeta(1+i|\alpha_i|-i |\alpha_j|)\right|.}  Applying these in \eqref{bound}, writing the sum of multiplicative functions as a product, and applying Merten's theorem one derives Theorem \ref{MP} for $\kappa \geq 4$.

\section{The Case of Weight Two}\label{two}
When $\kappa =2 $ the function $J\left(2\sqrt{nn_2n_3}/c\right)V_{1/2+\alpha}(n/q)$ is not continuous at $n=0$, nor is it of bounded variation when $\alpha \neq 0$, so the use of Poisson summation in section \ref{Initial} needs to be justified.  Let $Q(x)= J(2\sqrt{x})V_{1/2+\alpha}(x)$ and consider the toy case \est{\sum_{n\geq 1} \chi(n) S(n,1,c)Q(n)  = \sum_{a \pmod c} \chi(a) S(a,1,c) \sum_{n=0}^\infty Q(a+nc).}  The interior sum on the right hand side is periodic modulo $c$ so it has a Fourier series \est{\frac{1}{c} \sum_{m\in \Z} \widehat{Q}(m/c) e\left(\frac{ma}{c}\right).} It suffices to show that the Fourier series converges to the correct value.  Suppose first that it converges to some finite value for any $c\in \N$ and $a=0,\ldots,c-1$.  If $a=0$ both sides vanish because $\chi(0)=0$ and there is nothing to prove.  If $a\neq 0$ then $\sum Q(a+nc)$ is continuous at $a$, so the series is Ces\`aro summable to the correct value by Fej\'er's Theorem, hence pointwise summable to the correct value.  To see that the Fourier series converges to some finite value, we may calculate $\widehat{Q}$ explicitly using \eqref{MellinJ}, \eqref{Vdef}, \eqref{Ve}, and use Euler-Maclaurin summation.

Next note that the sums over $m_i$ do not converge absolutely in section \ref{Dual} when $\kappa=2$ (but the sum over $r$ does).  Absolute convergence was used in section \ref{Dual} to swap the sums over $m_i$ with the integrals from the Mellin inverses in Proposition \ref{Oscillatory}, so we need to instead justify the interchange by hand.

Let $\kappa=2$ and return to the sum defined by equation \eqref{total}.  Consider the lengthy expression formed by applying Lemma \ref{Gcalc} and formula \eqref{Hdecomp} via the decomposition \eqref{hkdecomp} to \eqref{total}.  However we now keep the functions $e_c(m_1m_2m_3)$ and $T_2(m_1,m_2,m_3,c)$ inside the sums over $m_i$.  The problematic interior sums over $m_i=k_in_i$ are \est{\sum_{\substack{(n_1,k/k_1)=1 \\ (n_1,qr'h/k)=1}} R_k(k_1n_1) \psi(n_1) \sum_{\substack{(n_2,k/k_2)=1 \\ (n_2,qh/k)=1}} R_k(k_2n_2)\psi(n_2) \sum_{\substack{(n_3,k/k_3)=1 \\ (n_3,qh/k)=1}}  R_k(k_3n_3)\psi(n_3) \\ \times e_c(k_1k_2k_3n_1n_2n_3)T_2(k_1n_1,k_2n_2,k_3n_3,c). }
We use the function $w_0(x)$ from section \ref{sec:statphase} which, recall, is smooth, identically 1 for $x<1/2$, and identically 0 for $x \geq 1$ to split the three sums over $n_i$ at a very large parameter $Z$.  Summing by parts in each of the three variables $n_i$, the tails of the series are equal to \es{\label{sbparts}- \iiint_{\R^3_{>0}} \Sigma(k,c,u_i) \frac{\partial^{3}}{\partial u_1 \partial u_2 \partial u_3} \left(T_2(k_1u_1,k_2u_2,k_3u_3,c)(1-w_0(u_1/Z)w_0(u_2/Z)w_0(u_3/Z))\right)\,du,} where \est{\Sigma(k,c,u_i) \eqdef \sum_{n_i \leq u_i} \prod_{i=1}^3 R_k(k_in_i)\psi(n_i)e_c(k_1k_2k_3n_1n_2n_3).}
From its definition as a Mellin inversion the function $T_2(k_1u_1,k_2u_2,k_3u_3,c)$ is smooth and satisfies \est{u_1^{j_1}u_2^{j_2}u_3^{j_3}\frac{\partial^{j_1+j_2+j_3}}{\partial u_1^{j_1} \partial u_2^{j_2} \partial u_3^{j_3} }T_2 \ll_{h,k,q,\eps} \min_{\lambda \in [\eps,1-\eps]} (u_1u_2u_3)^{\lambda-1+ \eps/10}r'^{-3\lambda-\eps/10}.}

\begin{lemma}\label{Sigmabound}We have the estimate \est{\Sigma(k,c,u_i) \ll_\eps \max \left((u_1u_2u_3)^\eps,r^{-3/4}(u_1u_2u_3)^{3/4+\eps}\right).}
\end{lemma}
\begin{proof}
We use formula \eqref{archarformula} from section \ref{Dual}: \est{e\left(\frac{m_1m_2m_3}{c}\right) = 1+ \frac{1}{2\pi i}\int_{(-3/4)} \frac{\Gamma(v)}{(-2\pi i )^v} \left(\frac{m_1m_2m_3}{c}\right)^{-v}\,dv.}  The integral converges absolutely.  We have then that $\Sigma(k,c,u_i)$ is given by \est{\sum_{n_i\leq u_i} R_k(k_in_i)\psi(n_i) + \frac{1}{2\pi i } \int_{(-3/4)} \frac{\Gamma(v)}{(-2\pi i )^v} c^{v} \sum_{n_i\leq u_i} R_k(k_in_i)\psi(n_i)n_i^{-v}.}  Lemma \ref{DirichletSeries} along with standard Perron formula and contour shifting techniques show that the first of these two terms is $\ll_{q,h,k,\eps} (u_1u_2u_3)^\eps,$ and the second is $\ll_{q,h,k,\eps} r^{-3/4}(u_1u_2u_3)^{3/4+\eps}$.  
\end{proof}
Evaluating the derivatives in \eqref{sbparts} produces eight terms, the most difficult of which is\es{\label{Tbound}k_1k_2k_3T_2^{(1,1,1)}(k_1u_1,k_2u_2,k_3u_3,c)(1-w_0(u_1/Z)w_0(u_2/Z)w_0(u_3/Z)) \\  \ll_{h,k,q,\eps} \min\left( (u_1u_2u_3)^{-3/2+\eps}r'^{-3/2}, (u_1u_2u_3)^{-11/6+\eps}r'^{-1/2}\right),} using $\lambda=1/2$ or $\lambda=1/6$.   
We may rigorously swap the main portion of the sums over $n_i$ with the Mellin inversion integrals from Proposition \ref{Oscillatory}.  The tails of the sums over $n_i$ are given by the integral \eqref{sbparts}.  This integral converges absolutely by Lemma \ref{Sigmabound} and \eqref{Tbound}, and moreover is $\ll_{q,h,k,\eps} r'^{-5/4}Z^{-1/12+\eps} $. Now taking the limit $Z\rightarrow \infty$ one shows by a Mellin transform argument that the smoothed partial main portion sums converge to the correct $L$-functions from Lemma \ref{DirichletSeries}, and the tails go to 0.  Thus we rigorously establish \eqref{big4} and Theorem \ref{Thm2} in the case $\kappa=2$. 

Finally, to derive Theorem \ref{MP} from Theorem \ref{Thm2} we shift the contour to $\real(s)=1/2-1/\log q$ and $\real(u_i) =1/4 \log q$.  Note that if we chose $\real(s)=1/2$ as we did in the cases $\kappa \geq 4$ then the integral \eqref{bound} would not converge.  With the choice $\real(s)=1/2-1/\log q$ we calculate \est{\int_{-\infty}^\infty \frac{dt}{(1+|t|)^{1+1/4\log q}} \ll \log q,} so we gain a single extra $\log q$ in the result when $\kappa =2$.  \end{proof}

\end{document}